\documentclass[11pt, a4paper]{amsart}
\usepackage[english]{babel}
\usepackage[utf8x]{inputenc}
\usepackage{amsmath,amsfonts,stackrel,amsthm,enumerate,subcaption}
\usepackage{graphicx,setspace,mathrsfs,placeins,mathtools,amssymb}
\usepackage[margin=1in]{geometry}

\newtheorem{definition}{Definition}
\newtheorem{theorem}{\bf Theorem}[section]
\newtheorem{remark}{\bf Remark}[section]

\newtheorem{lemma}{Lemma}[section]
\newtheorem{corollary}{Corollary}[section]
\newtheorem{example}{Example}[section]
\newtheorem{algorithm}{Algorithm}[section]
\usepackage{chngcntr}
\counterwithout{algorithm}{section}
\usepackage[colorinlistoftodos]{todonotes}
\title{Time-changed Poisson processes of order $k$}
\begin{document}
	\author{\small Ayushi S. Sengar}
	\address{\small Department of Mathematics\\
		Indian Institute of Technology Madras, Chennai 600036, INDIA.}
	\email{ma15d201@smail.iitm.ac.in}
	\author{\small A. Maheshwari}
	\address{\small Operations Management and Quantitative Techniques Area\\
		Indian Institute of Management Indore, Indore 453556, INDIA.}
	\email{adityam@iimidr.ac.in}
	\author{\small N. S. Upadhye}
	\address{\small Department of Mathematics\\
	Indian Institute of Technology Madras, Chennai 600036, INDIA.}
	\email{neelesh@iitm.ac.in}

%\textbf{\large{Poisson process of order k Time changed by L\'evy Subordinator and its inverse}}\\
\begin{abstract}
\noindent In this article, we study the Poisson process of order $k$ (PPoK) time-changed with an independent L\'evy subordinator and its inverse, which we call respectively, as TCPPoK-I and TCPPoK-II, through various distributional properties, long-range dependence and limit theorems for the PPoK and the TCPPoK-I. Further, we study the governing difference-differential equations of the TCPPoK-I for the case inverse Gaussian subordinator. Similarly, we study the distributional properties, asymptotic moments and the governing difference-differential equation of TCPPoK-II. As an application to ruin theory, we give a governing differential equation of ruin probability in insurance ruin using these processes. Finally, we present some simulated sample paths of both the processes.
\end{abstract}
\maketitle
\section{Introduction}

\noindent
Poisson process can be considered as a core object of applied probability, due to its simplicity and applicability in modelling count data, which led to evolution and generalization of Poisson processes in several directions. For example, non-homogeneous Poisson processes, Cox point processes, higher
dimensional Poisson processes, and for last two decades, the fractional (time-changed) variants of Poisson processes (see \cite{lask,mnv,beghinejp2009,sfpp} and references therein) have caught the attention of the researchers and a vast literature is available on this topic.
%\vskip 2ex
%\noindent
In particular, insurance models generally use Poisson process to model the arrival of claims with a limitation of not having more than one claim in a certain small time interval. However, the claim arrival in group insurance schemes may contain more than one claims. To overcome this difficulty, Kostadinova and Minkova (2012) \cite{Poiss-order-k} introduced a variant known as Poisson process of order $k$, which models the claim arrival in groups of size $k$, where the number of arrivals in a group is uniformly distributed over $k$ points. Further, in case of calamities, the time period between two claims may not have exponential distribution, as these are extreme events and can not be modelled by Poisson process of order $k$ (as defined in \cite{Poiss-order-k}). Hence there is a need to generate a new stochastic process which is a generalization of Poisson process of order $k$. 
\vskip 2ex
\noindent
Among various techniques to create a new process, the technique of subordination (or time-change) introduced by Bohner \cite{Bochner} has gained significant attention in recent years. The theory of subordinated processes is explored in detail in \cite{sato}. A subordinated stochatic process can be generated by replacing time of the original process with a stochastic process preferably having non-decreasing sample paths. In literature, various examples of subordinated processes are discussed, and shown to have interesting probabilistic properties and elegant connections to fractional calculus, see e.g. \cite{Allouba02, Allouba-Zheng01, Baeum-Meersch-Nane09, beghinejp2009, Hahn-Kobaya-Umarov11, Sato2001}. 
In paricular, recently, subordinated Poisson processes are studied by several authors (see  \cite{Kumar-TCPP, sfpp, fnbpfp, lrd2016, TCFPP-pub, OrsToa-Berns}).
Also, these processes are extensively used in several areas, such as physics  \cite{sub:phy1,sub:phy2,sub:phy3,sub:phy4,Barn-Niel97,Barn-Niel98}, ecology \cite{sub:eco}, biology \cite{sub:bio},
hydrology \cite{Meersch-Koz-Molz-Lu04} and finance \cite{subordinator:fin1,subordinator:fin2,subordinator:fin3,subordinator:fin4,Mandelbrot01}.
However, to the best of our knowledge, subordinated Poisson processes of order $k$ have not been explored. 
\vskip 2ex
\noindent
In this article, the main goal is to explore time-changed Poisson process of order $k$ with L\'evy subordinator (increasing L\'evy process) and its right-continuous inverse, as the transition probabilities of the new process with L\'evy subordinator allow us to have more than one arrivals in a small interval of time which is useful in modelling the count data occurring in lumps.

%The motivation to study time-changed processes is two fold. First, it studies the time-changed processes in general set up, which may also cover some fractional generalizations. Secondly, after time-changing the Poisson process with L\'evy subordinator the transition probability allows to have arbitrary arrivals in a small time interval as compared to one arrival in Poisson process, which may be very useful in modelling the count data of bursty nature. \\

\vskip 2ex
\noindent
The article is organized as follows. Section \ref{sec:prelim} deals with some preliminary definitions and results. In Sections \ref{sec:tcfppokI} and \ref{sec:tcfppokII}, Poisson process of order $k$ with a L\'evy subordinator and its right-continuous inverse are studied, respectively. The governing equations for the time-changed Poisson process of order $k$ are given in Section \ref{sec:dde}. Section \ref{sec:appl} discusses an application in ruin theory. Finally, some simulation algorithms to generate the sample paths of these processes are presented in Section \ref{sec:simu}.

\section{Preliminaries}\label{sec:prelim}
\noindent
In this section, we state some relevant definitions and results related to Poisson process of order $k$ and L\' evy subordinator.
\subsection{Poisson distribution of order $k$}
The early work on the distributions of order $k$ started with defining the notion geometric distribution of order $k$ (see \cite{phili83-geo}) which denotes the number of trials until the first occurrence of $k$ consecutive successes in a sequence of independent Bernoulli trials. The probability distribution of the sum of independent and identically distributed (IID) random variables having geometric distribution of order $k$ is called negative binomial distribution of order $k$ (NBoK). Let $Y_n$ denote NBoK, then the limiting distribution of $\{Y_n-kn\}$ as $n\to\infty$ is termed as Poisson distribution of order $k$  (PoK) (see \cite[Theorem 3.2]{phili83-geo}).

\begin{definition}
Let $x_1,x_2,\ldots,x_k$ be non-negative integers and  
 % \begin{equation}\label{symbol}
 $\zeta_k = x_1+x_2+\ldots+x_k$, $\Pi_k! = x_1!x_2!\ldots x_k!$ and  
 \begin{equation}\label{index}
\Omega(k,n):=\left\{{\bf x} = (x_1,x_2,\ldots,x_k)\big|x_1+2x_2+\ldots+kx_k=n\right\}.
\end{equation}
  %\end{equation}
Also, let $N^{(k)}$ follow PoK with rate parameter $\lambda>0$, then the probability mass function (\textit{pmf}) is given by
\begin{equation*}\label{def:poiss-dist-k-1}
\mathbb{P}[N^{(k)}=n]=\sum_{{\bf x} \in \Omega(k,n)} e^{-k\lambda } \frac{\lambda^{\zeta_k}}{\Pi_k!},~ n=0,1,\ldots.
\end{equation*}
\end{definition}
\noindent The probability generating function (\textit{pgf}) is given by (see \cite[Lemma 2.2]{phil1984})
\begin{equation}\label{pgf:poiss-dist-k-1}
G_{N^{(k)}}(s)=e^{-\lambda\left(k-\sum_{i=1}^{k}s^i\right)}.
\end{equation}
It is also known that (see \cite{Poiss-order-k}) the PoK has the following compound Poisson representation 
\begin{equation}\label{def:poiss-dist-k-2}
N^{(k)}\stackrel{d}{=} \sum_{i=0}^N X_i,
\end{equation}
where $N$ is Poisson random variable with rate parameter $k\lambda>0$, $X_0 \equiv 0$,  and $\{X_i\}_{i\ge 1}$is a sequence of IID discrete uniform random variable with \textit{pmf} given by $\mathbb{P}[X_i=j]=1/k,~ j=1,2,\ldots,k$, which is independent of $N$. Then the \textit{pgf} of $X_1$ is given by 
%\begin{equation}
$G_{X_{1}}(s)=\frac{s}{k}\frac{1-s^k}{1-s}$, $s\in(0,1).$
%\end{equation}
Therefore, the \textit{pgf} of $N^{(k)}$ given in \eqref{def:poiss-dist-k-2} is
\begin{equation}\label{pgf:poiss-dist-k-2}
G_{N^{(k)}}(s)=G_N(G_{X_1}(s))=e^{-k\lambda(1-G_{X_1}(s))}.
\end{equation}
It can be easily seen that {\textit{pgf}} obtained in \eqref{pgf:poiss-dist-k-1} and \eqref{pgf:poiss-dist-k-2} are same.
\subsection{The Poisson process of order $k$}
\noindent The Poisson process of order $k$ (PPoK) is introduced and studied by Kostadinova and Minkova (see \cite{Poiss-order-k}) which can be defined as follows. 
\begin{definition}
Let $\{N(t,k\lambda)\}_{t\geq 0}$ denote Poisson process with rate parameter $k\lambda>0$, $X_0 \equiv 0$, and  $\{X_i\}_{i \ge 1}$ be a sequence of IID discrete uniform random variables over $k$ points. Then the PPoK, $\{N^{(k)}(t,\lambda)\}_{t\geq 0}$, is defined (see \cite{Poiss-order-k}) as 
\begin{equation*}
N^{(k)}(t,\lambda)= \sum_{i=0}^{N(t,k\lambda)}X_i,
\end{equation*}
where $\{X_i\}_{i \ge 1}$ and $\{N(t, k \lambda)\}_{t \ge 0}$ are assumed to be independent.
\end{definition}
\noindent Henceforth, for brevity, the parameter $\lambda$ is suppressed and $N^{(k)}(t, \lambda)$ is written as $N^{(k)}(t)$, when no confusion arises.
\begin{remark}
For $k=1$, the distribution of $X_i$'$s$ degenerate to Dirac-delta distribution at 1 and $\{N^{(1)}(t)\}_{t\geq0}$ reduces to the Poisson process $\{N(t)\}_{t\geq0}$.
\end{remark}

\begin{remark}
The \textit{pgf} of $N^{(k)}(t)$ is 
%\begin{equation}
$G_{N^{(k)}(t)}(s)=\exp\left(-k\lambda t(1-G_{X_1}(s))\right)$,
%\end{equation}
where $G_{X_1}(s)=\frac{s}{k}\frac{1-s^k}{1-s}$ is the \textit{pgf} of $X_1$.
\end{remark}
%\noindent A process $ \{N^{(k)}(t)\}_{t\geq0}$ is said to be the Poisson process of order $k$ with rate parameter $\lambda>0$ %, denoted by $ P_k(\lambda)$, 
%,if its pmf is given by
%\begin{equation}
%\mathbb{P}(N^{(k)}(t)=n)~ = \sum_{{\bf x} \in \Omega(k,n)} e^{-k\lambda t} \frac{(\lambda t)^{\zeta_k}}{\Pi_k!},~t\geq 0,
%\end{equation} 
%where the sum is taken over index set $\Omega(k,n)$, defined in \eqref{index}.

\noindent The mean, variance and covariance function of the PPoK
are given by %(see \cite{Poiss-order-k})
\begin{align*}
\mathbb{E}[N^{(k)}(t)] &= \frac{k(k+1)}{2} \lambda t\\ 
\mbox{Var}[N^{(k)}(t)] &= \frac{k(k+1)(2k+1)}{6} \lambda t\\ 
 \mbox{Cov}[N^{(k)}(s),N^{(k)}(t)] &= \frac{k(k+1)(2k+1)}{6} \lambda \min(s,t).
\end{align*}
Also, observe that the transition probabilities of the PPoK $\{N^{(k)}(t)\}_{t\geq0}$ are given by %(see \cite{Poiss-order-k})
\begin{equation*}
\mathbb{P}[N^{(k)}(t+h)=n|N^{(k)}(t)=m]=\left\{
\begin{array}{ll}
1-k\lambda h+o(h)  & \mbox{if } n=m, \\
\lambda h+o(h) & \mbox{if } n=m+i,i=1,2,\ldots,k.
\end{array}
\right.
\end{equation*}

\noindent Let $p_m(t)=\mathbb{P}[N^{(k)}(t)=m],m=0,1,2,\ldots$ denote the \textit{pmf} of PPoK, then %(see \cite{Poiss-order-k})
\begin{align}
\frac{d}{dt}p_0(t)&=-k\lambda p_0(t)\nonumber,\\
\frac{d}{dt}p_m(t)&=-k\lambda p_m(t)+\lambda\sum_{j=1}^{m\wedge k}p_{m-j}(t),~m=1,2,\ldots,\label{PPoK-DE2}
\end{align}
with initial condition $p_0(0)=1 \text{ and } p_m(0)=0,m=1,2,\ldots$ and $m\wedge k:=\min\{m,k\}.$

\noindent
Next, note that the \textit{pgf} of  $\{N^{(k)}(t)\}_{t\geq0}$ satisfies the following differential equation %(see \cite{Poiss-order-k})
\begin{equation*}
\frac{\partial}{\partial t}G_{N^{(k)}(t)}(s)=-k\lambda[1-G_{X_1}(s)]G_{N^{(k)}(t)}(s),~~{\rm with} ~~ G_{N^{(k)}(0)}(0)=1.
\end{equation*}
%with the initial condition $$

\noindent The L\'evy exponent (characteristic exponent) (see \cite{ContTan2004}) of $\{N^{(k)}(t)\}_{t\geq0}$ is given by
\begin{equation*}
\psi(u)=\int_{-\infty}^\infty k\lambda(\exp(\iota uy)-1) \mu_{X_1}(dy).
\end{equation*}

\subsection{L\'evy subordinator}
A L\'evy subordinator (hereafter referred to as the subordinator) $\{D_{f}(t)\}_{t\geq0}$ is a non-decreasing L\'evy process and its Laplace transform (LT) (see \cite[Section 1.3.2]{appm}) has the form
\begin{equation}\label{subordinator-LT}
\mathbb{E}[e^{-s D_{f}(t)}]=e^{-tf(s)},% s>0,
%\end{equation}
\;{\rm where}\; 
%\begin{equation}\label{Bernstein-function}
f(s)=b s+\int_{0}^{\infty}(1-e^{-s x})\nu(dx),~b\geq0, s>0,
\end{equation}
is the Bernstein function (see \cite{Bernstein-book} for more details). 
Here $b$ is the drift coefficient and $\nu$ is a non-negative L\'evy measure on positive half-line satisfying 
\begin{equation*}
\int_{0}^{\infty}(x\wedge 1)\nu(dx)<\infty~~{\rm and}~~\nu([0,\infty))= \infty
\end{equation*}
which ensures that the sample paths of $D_{f}(t)$ are almost surely $(a.s.)$  strictly increasing.
%\vskip 2ex
%\noindent
Also, the first-exit time of $\{D_f(t)\}_{t\geq0}$ is defined as
$E_{f}(t)=\inf\{r\geq 0:D_{f}(r)>t\}$,% ~~~t\geq 0,
which is the right-continuous inverse of the subordinator $\{D_f(t)\}_{t\geq 0}$.

\begin{remark} \label{rem:levysub}
Note that a L\' evy subordinator is a class of subrodinators, which is useful in generating various subordinated stochastic processes in general. Next, we include some well-known examples of L\'evy subordinators with drift coefficient $b = 0$ which are used later in the article.
\begin{enumerate}[(i)]
\item Let the L\'evy measure be $\nu(dx) = \frac{pe^{-\alpha x}}{x} dx, ~x>0,p>0,\alpha>0$ then using (\ref{subordinator-LT}), we get the Gamma subordinator $\{Y(t)\}_{t \ge 0}$ with Bernstein function $f(s)= p\log(1+\frac{s}{\alpha})$ (see \cite{ContTan2004}, p. 115).
\item Let the L\'evy measure be $\nu(dx) = c\frac{e^{-\mu x}}{x^{\alpha+1}}dx, ~x>0,c>0,\mu>0, 0<\alpha<1$ then using (\ref{subordinator-LT}), we get the Tempered $\alpha$-stable subordinator $D_{\alpha}^{\mu}(t)$ with Bernstein function $f(s)= (s+\mu)^{\alpha}-\mu^{\alpha}$ (see \cite{ContTan2004}, p. 115).
\item Let the L\'evy measure be $\nu(dx) = \frac{\delta}{\sqrt{2\pi x^3}}e^{\frac{-\gamma^2 x}{2}} dx,~x>0,\gamma>0,\delta>0 $ then using (\ref{subordinator-LT}), we get the Inverse Gaussian subordinator $G(t)$ with Bernstein function $f(s)=\delta(\sqrt{2s+\gamma^2}-\gamma)$ (see \cite{Kumar-Hitting}).
\end{enumerate}
\end{remark}
\section{Time-changed Poisson process of order $k$ - I}\label{sec:tcfppokI}

\noindent In this section, we consider the PPoK with a subordinator $\{D_f(t)\}_{t \geq 0}$, satisfying $\mathbb{E}[D_f^{\rho}(t)] < \infty$ for all $ \rho > 0 $, which can be defined as follows.
\begin{definition} The time-changed PPoK of Type-I (TCPPoK-I) is defined as
$$ \{Q_f^{(1)}(t)\} = \{N^{(k)}(D_f(t))\} ,~~ t \geq 0$$
where $ \{N^{(k)}(t)\}_{t\geq0}$ is the PPoK and is independent of the subordinator $\{ D_f(t)\}_{t\geq0}$.
\end{definition}
\noindent Next, we derive some properties of the TCPPoK-I. Let us first compute its {\textit{pmf}}.
\begin{theorem}\label{thm:tcppk-I-pmf} Let the Bernstein function $f(s)$, as defined in (\ref{subordinator-LT}), be such that $ \mathbb{E}[D_f^{\rho}(t)] < \infty$ for all $ \rho > 0 $. Then, the \textit{pmf} of the TCPPoK-I is given by
\begin{equation}\label{pmf-tcppok}
 P[Q_f^{(1)}(t) = n] = \sum_{{\bf x} \in \Omega(k,n)} \frac{ \lambda ^{\zeta_k}}{\Pi_k!} \mathbb{E}\left[  e^{-k\lambda D_f(t)} D_f^{\zeta_k}(t)\right],~~n=0,1,2,\ldots.
\end{equation}
\begin{proof}

Let $g_f(y,t)$ be  the probability density function (\textit{pdf}) of L\'evy subordinator. Then
\begin{align*}
P[Q_f^{(1)}(t) = n] = P[N^{(k)}(D_f(t)) = n] 
=& \int_0^{\infty} P[N^{(k)}(D_f(t)) = n | D_f(t)]g_f(y,t)dy 
%\\ & & where ~g_f(y,t)~ is~ the ~pdf~ of~ L\'evy ~subordinator~.
\\ 
 %=& \int_0^{\infty} P[N^{(k)}(y) = n | D_f(t)=y]g_f(y,t)dy \\ 
 %=& \int_0^{\infty} P[N^{(k)}(y) = n]g_f(y,t)dy \\ 
 =& \int_0^{\infty} \sum_{{\bf x} \in \Omega(k,n)} \frac{ e^{-k\lambda y} (\lambda y)^{\zeta_k}}{\Pi_k!} g_f(y,t)dy \\
%  =&  \sum_{{\bf x} \in \Omega(k,n)} \frac{ \lambda^{\zeta_k}}{\Pi_k!} \int_0^{\infty}  e^{-k\lambda y}  y^{\zeta_k} g_f(y,t)dy \\
=&  \sum_{{\bf x} \in \Omega(k,n)} \frac{ \lambda ^{\zeta_k}}{\Pi_k!} \mathbb{E}\left[  e^{-k\lambda D_f(t)} D_f^{\zeta_k}(t)\right],
%\\  
% =& \sum_{\substack{x_1,x_2,...,x_k\geq 0\\
 %		\sum_{j=1}^{k}jx_j=n}} \frac{ (\lambda )^{x_1+x_2+..+x_k}}{x_1!x_2!...x_k!} \int_0^{\infty} \sum_{m=0}%^{\infty} \frac{(-k\lambda y)^m}{m!} ( y)^{x_1+x_2+..+x_k} g_f(y,t)dy \\ =& \sum_{\substack{x_1,x_2,...,x_k\geq 0\\%%%\sum_{j=1}^{k}jx_j=n}} \frac{ (\lambda )^{x_1+x_2+..+x_k}}{x_1!x_2!...x_k!}  \sum_{m=0}^{\infty} \frac{(-k\lambda %)^m}{m!} \int_0^{\infty} ( y)^{x_1+x_2+..+x_k+m} g_f(y,t)dy \\ =& \sum_{\substack{x_1,x_2,...,x_k\geq 0\\\sum_{j=1}%^{k}jx_j=n}} \frac{ (\lambda )^{x_1+x_2+..+x_k}}{x_1!x_2!...x_k!}  \sum_{m=0}^{\infty} \frac{(-k\lambda )^m}{m!} %\mathbb{E}[D_f^{x_1+x_2+..+x_k+m}(t)], 
\end{align*}
 which completes the proof.
\end{proof}
%\noindent 
%{\color{red} Q. Any relation between $m$ and $n$? \\ \color{blue} A. There is no relation between m and n. \\
%\color{red}Please add proof of this Theorem.}
\end{theorem}
\begin{corollary}
The {\textit{pmf}} of the TCPPoK-I satisfies the normalizing condition $$ \sum_{n=0}^{\infty} P[Q_f^{(1)}(t) = n]=1.$$
\end{corollary}
\begin{proof}
We first prove this result for the case $k = 2$. From (\ref{pmf-tcppok}) we have
$$ \sum_{n=0}^{\infty}  P[Q_f^{(1)}(t) = n] = \sum_{n=0}^{\infty} \sum_{{\bf x} \in \Omega(2,n)} \frac{\lambda ^{\zeta_2}}{\Pi_2!}\mathbb{E}\left[  e^{-2\lambda D_f(t)} D_f^{\zeta_2}(t)\right].$$
%put $ x_1 = n-2x_2$,
%$$  =  \sum_{n=0}^{\infty} \sum_{x_2=0}^{\frac{n}{2}} \frac{(\lambda )^{n-x_2}}{(n-2x_2)!x_2!}\sum_{m=0}^{\infty} \frac{(-2\lambda)^m}{m!} \mathbb{E}[D_f^{n-x_2+m}(t)] $$
%$$ = \sum_{n=0}^{\infty} \sum_{x_2=0}^{\frac{n}{2}} \frac{(\lambda )^{n-x_2}}{(n-2x_2)!x_2!}\sum_{l=n-x_2}^{\infty} \frac{(-2\lambda)^{l-n+x_2}}{(l-n+x_2)!} \mathbb{E}[D_f^{l}(t)] $$
%{\color{red} This step is not clear to me. Kindly provide more details.}\\
%{\color{blue} Now, see the proof.}\\\\
Set $ x_i = n_i ~ i=1,2$ and $n=x+\sum_{i=1}^2(i-1)n_i $ in the above expression. Then
%{\color{red} In second step, $n=x+n_2$, so the first summation should be $\sum_{x+n_2=0}^{\infty}$, which is not free of $n_2$, so how  the term $\frac{\lambda^x}{x!}$ can come out of the second summation?}\\
\begin{align*}
\sum_{n=0}^{\infty}  P[Q_f^{(1)}(t) = n] =& \sum_{x+n_2=0}^{\infty} \sum_{\substack{n_1,n_2\geq 0\\n_1+n_2  = x}} \frac{\lambda^{n_1+n_2}}{n_1!n_2!}\mathbb{E}\left[  e^{-2\lambda D_f(t)} D_f^{n_1+n_2}(t)\right] \\ 
%=& \sum_{x+n_2=0}^{\infty}  \sum_{n_2=0}^x \frac{\lambda^x}{(x-n_2)!n_2!}\mathbb{E}\left[  e^{-2\lambda D_f(t)} D_f^{x}(t)\right]  \\
% =& \sum_{x=0}^{\infty} \sum_{n_2=0}^{\infty} \frac{\lambda^x}{(x-n_2)!n_2!} \mathbb{E}\left[  e^{-2\lambda D_f(t)} D_f^{x}(t)\right] 1_{\{ n_2 \leq x\}}\\
 %=& \sum_{x=0}^{\infty} \sum_{n_2=0}^{x} \frac{\lambda^x}{(x-n_2)!n_2!} \mathbb{E}\left[  e^{-2\lambda D_f(t)} D_f^{x}(t)\right]  \\ 
 =& \sum_{x=0}^{\infty} \frac{(2\lambda)^x}{x!} \mathbb{E}\left[  e^{-2\lambda D_f(t)} D_f^{x}(t)\right] \text{(using binomial theorem)} \\
 =& \int_0^{\infty}e^{-2\lambda y} \sum_{x=0}^{\infty} \frac{(2\lambda)^x}{x!} y^x g_f(y,t)dy \\ =& \int_0^{\infty}e^{-2\lambda y} e^{2\lambda y}  g_f(y,t)dy  =  \int_0^{\infty} g_f(y,t)dy  = 1.
\end{align*}
Using similar arguments one can prove for higher values of $k$.
%The proof of values for $k>2$ can be shown in a similar way.
%$$ =  \sum_{l=0}^{\infty}\sum_{x_2=0}^{\infty}\sum_{n=x_2}^{\infty} \frac{(\lambda )^{n-x_2}}{(n-2x_2)!x_2!}  \frac{(-2\lambda)^{l-n+x_2}}{(l-n+x_2)!} \mathbb{E}[D_f^{l}(t)]$$
%$$ = \sum_{l=0}^{\infty} \mathbb{E}[D_f^{l}(t)]\sum_{x_2=0}^{\infty}\sum_{k=0}^{\infty} \frac{(\lambda )^{k}}{(k-x_2)!x_2!}  \frac{(-2\lambda)^{l-k}}{(l-k)!} $$
%$$ = \sum_{l=0}^{\infty} \mathbb{E}[D_f^{l}(t)] \frac{\lambda^l}{l!}\sum_{x_2=0}^{\infty}\sum_{k=0}^{\infty}\frac{(-2)^{l-k}1^{k-x_2}1^{x_2}l!}{(k-x_2)!x_2! (l-k)!}=1.$$ 
\end{proof}
\noindent Using simple algebraic calculations, one can see that the transition probabilities of the TCPPoK-I $\{Q_f^{(1)}(t)\}_{t\geq 0}$ are given by
\begin{equation}\label{tran-prob} \mathbb{P}[Q_f^{(1)}(t+h)=n|Q_f^{(1)}(t)=m] = \begin{cases} 1-hf(k \lambda)+ o(h),&n=m \\ 
&\\-h\left(\sum\limits_{{\bf x} \in \Omega(k,i)} \frac{ (- \lambda )^{\zeta_k}}{\Pi_k!}f^{(\zeta_k)}(k\lambda)\right)  +o(h),&n=m+i,~i=1,2,\ldots\end{cases}, \end{equation}
where $f(k\lambda)$ is the Bernstein function.\\
\noindent Further, we present some interesting examples for the TCPPoK-I.

\begin{example}[Negative Binomial process of order $k$] It is known that negative binomial process can be obtained by subordinating the Poisson process with gamma process (see \cite{fnbpfp}). In a similar spirit, we can define the negative binomial process of order $k$ by subordinating PPoK with an independent gamma process $\{Y(t)\}_{t\geq0 }$ as defined in Remark \ref{rem:levysub}(i) and its \textit{pmf} is given by
$$ \mathbb{P}[N^{(k)}(Y(t)) = n] =  \sum_{{\bf x} \in \Omega(k,n)} \frac{\lambda^{\zeta_k}}{\Pi_k!}\sum_{m=0}^{\infty} \frac{(-k\lambda)^m}{m!} \frac{\Gamma(pt+\zeta_k+m)}{\alpha^{\zeta_k+m}\Gamma(pt)},~ n=0,1,2,\ldots.$$
%For $ k=2$, it reduces to
%\begin{align*}
%\mathbb{P}[N^{(2)}(Y(t)) = n] &= \int_0^{\infty} \sum_{{\bf x} \in \Omega(2,n)} \frac{e^{-2\lambda y}(\lambda y )^{\zeta_2}}{\Pi_2!}\frac{\alpha^{pt}}{\gamma(pt)}y^{pt-1}e^{-\alpha y}dy, 
%\shortintertext{after solving this, we get}
%& = \sum_{{\bf x} \in \Omega(2,n)} \frac{(\zeta_2+pt-1)!}{\Pi_2!(pt-1)!} \left(\frac{\lambda}{\alpha+2\lambda}\right)^{x_1+x_2} \left(\frac{\alpha}{\alpha+2\lambda}\right)^{pt}. 
%\end{align*}
\end{example}

\begin{example}[Poisson-tempered $\alpha$-stable process of order $k$]
Let $ \{D_{\alpha}^{\mu}(t)\}_{t\geq 0},~ \mu >0, ~0<\alpha <1$ be the tempered $ \alpha$-stable subordinator  as defined in Remark \ref{rem:levysub}(ii).
% The \textit{pdf} of tempered $\alpha$-stable subordinator is given by (see \cite{fnbpfp})
%$$ g_{\mu}(x,t) = e^{-\mu x + \mu^{\beta}t} g(x,t),~x >0,$$
%where $g(x,t)$ is pdf of the $\alpha$-stable subordinator $ D_{\alpha}(t)$.
Then \textit{pmf} of the Poisson-tempered $\alpha$-stable of order $k$ is given by
$$ \mathbb{P}[N(D_{\alpha}^{\mu}(t)) = n] = \sum_{{\bf x} \in \Omega(k,n)} \frac{(\lambda )^{\zeta_k}}{\Pi_k!} e^{\mu^{\alpha}t} \sum_{m=0}^{\infty} \frac{(-k\lambda)^m}{m!} \mathbb{E}[(D_{\alpha}(t))^{\zeta_k+m}e^{-\mu D_{\alpha}(t)}], n=0,1,2,\ldots. $$
\end{example}
\begin{example} [Poisson-inverse Gaussian process of order $k$]
Let $ \{G(t)\}_{t \geq 0}$ be the inverse Gaussian subordinator  as defined in Remark \ref{rem:levysub}(iii). The  moments of $\{G(t)\}_{t\geq0} $  are given by (see \cite{fnbpfp})
$$ \mathbb{E}[G^q(t)] = \sqrt{\frac{2}{\pi}}\delta \left(\frac{\delta t}{\gamma}\right)^{q-\frac{1}{2}}te^{\delta \gamma t}K_{q-\frac{1}{2}}(\delta \gamma t),~~\delta, \gamma >0,~ t \geq 0,~ q \in (-\infty,\infty),$$
where $ K_{\nu}(z)$ is the modified Bessel function of third kind with index $ \nu $, defined by
$$ K_{\nu}(\omega) = \frac{1}{2} \int_0^{\infty} x^{\nu-1}e^{\frac{-1}{2}\omega(x+x^{-1})}dx, ~ \omega >0. $$
\\Using the above expression, we get the following
$$  \mathbb{E}[G^{\zeta_k+m}(t)] = \sqrt{\frac{2}{\pi}}\delta \left(\frac{\delta t}{\gamma}\right)^{(\zeta_k+m)-\frac{1}{2}}t e^{\delta \gamma t}K_{(\zeta_k+m)-\frac{1}{2}}(\delta \gamma t), $$
where $ \delta, \gamma >0,~ t \geq 0.$ Substituting above values of moments in Theorem \ref{thm:tcppk-I-pmf}, we get the  {\textit{pmf}} of Poisson-inverse Gaussian process of order $k$.
\end{example}
\noindent
Next, we discuss some distributional properties of TCPPoK-I.
\begin{theorem} \label{theorem:distribution}
 Let $ 0<s \leq t < \infty$, then the mean and covariance function of TCPPoK-I are as follows
	\begin{enumerate}[(i)]
		\item $  \mathbb{E}[Q_f^{(1)}(t)]= \frac{k(k+1)}{2} \lambda \mathbb{E}[D_f(t)], $
	%	\item $  \text{Var}[Q_f^{(1)}(t)] = \frac{k(k+1)(2k+1)}{6}\lambda \mathbb{E}[D_f(t)] + ( \frac{k(k+1)}{2} \lambda )^2 Var[D_f(t)], $
		\item $ \text{Cov}[Q_f^{(1)}(s),Q_f^{(1)}(t)] = \frac{k(k+1)(2k+1)}{6}\lambda \mathbb{E}[D_f(s)] + ( \frac{k(k+1)}{2} \lambda )^2 \text{Var}[D_f(s)],$
	\end{enumerate} 
\begin{proof}
Let $g_f(y,t)$ be  the \textit{pdf} of the L\'evy subordinator $\{D_f(t)\}_{t\geq 0}$. Then
%\begin{align*}
$$ \mathbb{E}[Q_f^{(1)}(t)] = \mathbb{E}[N^{(k)}(D_f(t))]%\\
 = \mathbb{E}[\mathbb{E}[N^{(k)}(D_f(t))|D_f(t)]] % \\
% =& \int_0^{\infty} \mathbb{E}[N^{(k)}(y,\lambda)]g_f(y,t) dy\\
=% \frac{k(k+1)}{2} \lambda \int_0^{\infty} y g_f(y,t) dy \\=&
 \frac{k(k+1)}{2} \lambda \mathbb{E}[D_f(t)],$$
%\end{align*}
which proves Part (i). 

\noindent
Now, we derive the expression for covariance of TCPPoK-I. 
\ifx
\begin{align*}
\text{Var}[Q_f^{(1)}(t)] =&  \text{Var}[N^{(k)}(D_f(t))]\\ =& \mathbb{E}[(N^{(k)}(D_f(t)) - \mathbb{E}[N^{(k)}(D_f(t))])^2] \\ =&  \int_0^{\infty}  \mathbb{E}[(N^{(k)}(D_f(t)) - \mathbb{E}[N^{(k)}(D_f(t))])^2 | D_f(t) = y] g_f(y,t) dy  \\=& \int_0^{\infty}  \mathbb{E}[(N^{(k)}(y) - \mathbb{E}[N^{(k)}(D_f(t))])^2 | D_f(t) = y] g_f(y,t) dy \\ =& \int_0^{\infty}  \mathbb{E}[(N^{(k)}(y) - \mathbb{E}[N^{(k)}(D_f(t))])^2] g_f(y,t) dy \\ =& \int_0^{\infty}  \mathbb{E}[( \{N^{(k)}(y) - \frac{k(k+1)}{2}\lambda y \} +\{ \frac{k(k+1)}{2}\lambda y - \mathbb{E}[N^{(k)}(D_f(t))] \} )^2] g_f(y,t) dy \\ =& \int_0^{\infty}  \mathbb{E}[( N^{(k)}(y) - \frac{k(k+1)}{2}\lambda y )^2] g_f(y,t) dy + \\ & \int_0^{\infty} \mathbb{E}[( \frac{k(k+1)}{2}\lambda y - \mathbb{E}[N^{(k)}(D_f(t))])^2] g_(y,t)dy \\ & + 2 \int_0^{\infty}  \mathbb{E}[( N^{(k)}(y) - \frac{k(k+1)}{2}\lambda y )(\frac{k(k+1)}{2}\lambda y - \mathbb{E}[N^{(k)}(D_f(t))])] g_f(y,t) dy \\=& \int_0^{\infty}  \mathbb{E}[( N^{(k)}(y) - \frac{k(k+1)}{2}\lambda y )^2] g_f(y,t) dy + \\ & \int_0^{\infty} \mathbb{E}[( \frac{k(k+1)}{2}\lambda y -\frac{k(k+1)}{2}\lambda \mathbb{E}[D_f(t)])^2] g_f(y,t) dy + 2(0) \\ =& \int_0^{\infty} \text{Var}[N_k(y)] g_f(y,t)dy + ( \frac{k(k+1)}{2} \lambda )^2 \int_0^{\infty}\mathbb{E}[(y-\mathbb{E}[D_f(t)])^2] g_f(y,t) dy  \\ =& \frac{k(k+1)(2k+1)}{6}\lambda \int_0^{\infty} y g_f(y,t) dy + ( \frac{k(k+1)}{2} \lambda )^2 \text{Var}[D_f(t)] \\=& \frac{k(k+1)(2k+1)}{6}\lambda \mathbb{E}[D_f(t)] + ( \frac{k(k+1)}{2} \lambda )^2 \text{Var}[D_f(t)]
\end{align*}
\fi
%\begin{align}
%\text{Cov}[Q_f^{(1)}(s),Q_f^{(1)}(t)] =& \mathbb{E}[Q_f^{(1)}(s)Q_f^{(1)}(t)]- \mathbb{E}[Q_f^{(1)}(s)] \mathbb{E}[Q_f^{(1)}(t)] \nonumber\\ =& \mathbb{E}[N^{(k)}(D_f(s)) N^{(k)}(D_f(t))] -  \mathbb{E}[N^{(k)}(D_f(s))] \mathbb{E}[ N^{(k)}(D_f(t))].\label{cov-tcppk-1}
%\end{align}
First, we evaluate $\mathbb{E}[Q_f^{(1)}(s)Q_f^{(1)}(t)] $. %For this we need to obtain the expression for $ \mathbb{E}[(N^{(k)}(D_f(s)))^2] $. Consider
%\begin{align*}
%\mathbb{E}[(N^{(k)}(D_f(s)))^2]=& \int_0^{\infty}\mathbb{E}[(N^{(k)}(D_f(s)))^2|D_f(s)]P(D_f(s)\in dy)\\
%=& \int_0^{\infty} \mathbb{E}[(N^{k}(y))^2] g_f(y,s)dy\\
%=& \int_0^{\infty} \left( \text{Var}[N^k(y)] +\mathbb{E}[N^k(y)]^2\right)g_f(y,s)dy\\
%%=& \frac{k(k+1)(2k+1)}{6}\lambda \int_0^{\infty}yg_f(y,s)dy + \left(\frac{k(k+1)\lambda}{2} \right)^2\int_0^{\infty} y^2 g_f(y,s)dy\\
%=&  \frac{k(k+1)(2k+1)}{6}\lambda \mathbb{E}[D_f(s)] + \left(\frac{k(k+1)\lambda}{2} \right)^2 \mathbb{E}[(D_f(s))^2].
%\end{align*} 
%Now
\begin{align*}
\mathbb{E}[Q_f^{(1)}(s)Q_f^{(1)}(t)] =& \mathbb{E}[N^{(k)}(D_f(s)) N^{(k)}(D_f(t))] \\ =& \mathbb{E}[N^{(k)}(D_f(s))\{ N^{(k)}(D_f(t))- N^{(k)}(D_f(s))\}] + \mathbb{E}[(N^{(k)}(D_f(s)))^2] \\
=& \mathbb{E}[N^{(k)}(D_f(s))] \mathbb{E}[N^{(k)}(D_f(t))-N^{(k)}(D_f(s))] + \mathbb{E}[(N^{(k)}(D_f(s)))^2] \\
%=&   \mathbb{E}[N^{(k)}(D_f(s))] \mathbb{E}[N^{(k)}(D_f(t))-N^{(k)}(D_f(s))] + \text{Var}[N^{(k)}(D_f(s))]+ \mathbb{E}[N^{(k)}(D_f(s))]^2 \\
%=& \mathbb{E}[N^{(k)}(D_f(s))] \mathbb{E}[N^{(k)}(D_f(t)-D_f(s))] + \text{Var}[N^{(k)}(D_f(s))]+ \mathbb{E}[N^{(k)}(D_f(s))]^2 \\
=& \mathbb{E}[N^{(k)}(D_f(s))] \mathbb{E}[N^{(k)}(D_f(t-s))] + \mathbb{E}[(N^{(k)}(D_f(s)))^2] \\ 
=& \frac{k(k+1)}{2}\lambda \mathbb{E}[D_f(s)]\frac{k(k+1)}{2}\lambda \mathbb{E}[D_f(t-s)] + \\
& \frac{k(k+1)(2k+1)}{6}\lambda \mathbb{E}[D_f(s)] + \left(\frac{k(k+1)\lambda}{2} \right)^2 \mathbb{E}[(D_f(s))^2],
\end{align*}
where the last equality follows from the fact that $$\mathbb{E}[(N^{(k)}(D_f(s)))^2]= \frac{k(k+1)(2k+1)}{6}\lambda \mathbb{E}[D_f(s)] + \left(\frac{k(k+1)\lambda}{2} \right)^2 \mathbb{E}[(D_f(s))^2].$$
Therefore, we get
\begin{align*}
Cov[Q_f^{(1)}(s),Q_f^{(1)}(t)] =& \mathbb{E}[Q_f^{(1)}(s)Q_f^{(1)}(t)]- \mathbb{E}[Q_f^{(1)}(s)]\mathbb{E}[Q_f^{(1)}(t)]\\
%& \frac{k(k+1)}{2}\lambda \mathbb{E}[D_f(s)]\frac{k(k+1)}{2}\lambda \mathbb{E}[D_f(t-s)] + \\ & \frac{k(k+1)(2k+1)}{6}\lambda \mathbb{E}[D_f(s)] + \left(\frac{k(k+1)\lambda}{2} \right)^2 \mathbb{E}[(D_f(s))^2] - \\ 
%& \frac{k(k+1)}{2}\lambda \mathbb{E}[D_f(s)]\frac{k(k+1)}{2}\lambda \mathbb{E}[D_f(t)] \\ 
=& \frac{k(k+1)(2k+1)}{6}\lambda \mathbb{E}[D_f(s)] + \left( \frac{k(k+1)}{2} \lambda \right)^2 \text{Var}[D_f(s)].
\end{align*}
which completes the proof of Part (ii). To get the expression of variance of the TCPPoK-I, we can put $s=t$ in the Part (ii).
\end{proof}
\end{theorem}
\begin{remark}
From Theorem \ref{theorem:distribution}, it is clear that $Var[Q_f^{(1)}(t)]> \mathbb{E}[Q_f^{(1)}(t)]$.Therefore, the index of dispersion $ I(t) := Var[Q_f^{(1)}(t)]/\mathbb{E}[Q_f^{(1)}(t)]$ (see \cite{TCFPP-pub} for more details) is greater than 1. Hence, we conclude that TCPPoK-I exhibits overdispersion. 
\end{remark}
\ifx	
\subsection*{Index of Dispersion}
The index of dispersion for a counting process $ X(t)$ is defined by (see \cite{TCFPP-pub})
$$ I(t) = \frac{\text{Var}[X(t)]}{\mathbb{E}[X(t)]}.$$
%and $ \text{Var}[Q_f^{(1)}(t)] - \mathbb{E}[Q_f^{(1)}(t)]>0$. 
The stochastic process $\{X(t)\}_{t\geq0}$ is said to be overdispersed if $I(t)>1$ for all $t\geq0$ nn.(see \cite{BegClau14})\\
Now, we claim that the TCPPoK-I $\{Q_{f}^{(1)}(t)\}_{t\geq0}$ exhibits overdispersion. Since the mean of the TCPPoK-I $\{Q_{f}^{(1)}(t)\}_{t\geq0}$ is nonnegative, hence it sufficient  to show that Var$[Q_{f}^{(1)}(t)]-\mathbb{E}[Q_{f}^{(1)}(t)]>0$.  To see this, consider
\begin{align*}
 \text{Var}[Q_f^{(1)}(t)] - \mathbb{E}[Q_f^{(1)}(t)] =& \frac{k(k+1)(2k+1)}{6}\lambda \mathbb{E}[D_f(t)] + \left( \frac{k(k+1)}{2} \lambda \right)^2 \text{Var}[D_f(t)]\\~~~~&~~~~ - \frac{k(k+1)}{2}\lambda \mathbb{E}[D_f(t)] \\ =& \frac{k(k+1)}{2}\lambda \mathbb{E}[D_f(t)] \left[\frac{2k+1}{3}-1\right] + \left( \frac{k(k+1)}{2} \lambda \right)^2 \text{Var}[D_f(t)].
 \end{align*}
  The L\'evy subordinator $ D_f(t)$ is strictly increasing \textit{a.s.} (see \cite{appm}), therefore $\mathbb{E}[D_f(t)]$ is non negative and $\left[\frac{2k+1}{3}-1\right] \geq 0,$ for  $k=1,2,3,\ldots$. Hence, the above expression is always positive which implies that TCPPoK-I $\{Q_{f}^{(1)}(t)\}_{t\geq0}$ is overdispersed.
\fi
\subsection{Long-range dependence}
Now we discuss the long-range dependence (LRD) property of the TCPPoK-I. We first need the following definitions.
\begin{definition}\label{Def:asym-equal}
Let $ f(x) $ and $ g(x)$ be positive functions. We say that $f(x)$ is asymptotically equal to $ g(x) $, written as $ f(x) \sim g(x), ~as~~x \rightarrow \infty$, if
$$ \lim_{x \rightarrow \infty} \frac{f(x)}{g(x)} = 1$$ 
\end{definition}

\begin{definition} \label{Def:LRD} (see \cite{lrd2016})
Let $ 0\leq s<t$ and $s$ be fixed. Assume a stochastic process $ \{ X(t)\}_{t \geq 0}$ has the correlation function $ Corr[X(s),X(t)]$ that satisfies
$$ c_1(s)t^{-d} \leq Corr[X(s),X(t)] \leq c_2(s)t^{-d}, $$
for large $ t,d >0,~ c_1(s)>0 ~and~ c_2(s)>0$. That is,
$$ \lim_{t \rightarrow \infty} \frac{Corr[X(s),X(t)]}{t^{-d}} = c(s)$$
for some $ c(s)>0$ and $ d>0$. We say that $ X(t)$ has the long-range dependence (LRD) property if $d \in (0,1)$ and short-range dependence (SRD) property if $ d \in (1,2)$.
\end{definition}

\noindent Now, we show that the TCPPoK-I has the LRD property.
\begin{theorem}
Let $ D_f(t) $ be such that $ \mathbb{E}[D_f(t)] \sim k_1 t^{\rho} $ and $ \mathbb{E}[(D_f(t))^2] \sim k_2 t^{2 \rho} $ for some $ 0< \rho <1 $, and positive constant $ k_1$ and  $k_2 $ with $ k_2 \geq k_1^2 $. Then the TCPPoK-I has the LRD property. 
\end{theorem}
\begin{proof}
 Let $ 0\leq s<t< \infty $, we have that
\begin{align*}
\text{Var}[Q_f^{(1)}(t)] =& 
%\frac{k(k+1)(2k+1)}{6}\lambda \mathbb{E}[D_f(t)] + \left( \frac{k(k+1)}{2} \lambda \right)^2 \text{Var}[D_f(t)] \\
\frac{k(k+1)(2k+1)}{6}\lambda \mathbb{E}[D_f(t)] + \left( \frac{k(k+1)}{2} \lambda \right)^2 \left( \mathbb{E}[D_f(t)^2] - \mathbb{E}[D_f(t)]^2 \right) \\ \sim 
& \frac{k(k+1)(2k+1)}{6}\lambda k_1 t^{\rho} + \left( \frac{k(k+1)}{2} \lambda \right)^2 \left( k_2 t^{2\rho} - (k_1 t^{\rho})^2 \right) \\ %\frac{k(k+1)(2k+1)}{6}\lambda k_1 t^{\rho} + \left(\frac{k(k+1)}{2} \lambda \right)^2 t^{2\rho} (k_2 - k_1^2) \\
\sim & \left( \frac{k(k+1)}{2} \lambda \right)^2 t^{2\rho} ( k_2 - k_1^2 ) ~~~~(\text{using~Definition \ref{Def:asym-equal}}), \\ =& d_1 t^{2 \rho},
\end{align*}
where $ d_1 = \left(\frac{k(k+1)}{2} \lambda \right)^2 (k_2 - k_1^2)$. Now, we study the asymptotic behavior of the correlation function
\begin{align*}
\text{Corr}[Q_f^{(1)}(s),Q_f^{(1)}(t)] =& \frac{\text{Cov}[Q_f^{(1)}(s),Q_f^{(1)}(t)]}{\sqrt{\text{Var}[Q_f^{(1)}(s)]\text{Var}[Q_f^{(1)}(t)]}} \\
 %\frac{\frac{k(k+1)(2k+1)}{6}\lambda \mathbb{E}[D_f(s)] + \left( \frac{k(k+1)}{2} \lambda \right)^2 \text{Var}[D_f(s)]}{\sqrt{\text{Var}[Q_f^{(1)}(s)]\text{Var}[Q_f^{(1)}(t)]}} \\ 
\sim & \frac{k(k+1)(2k+1)\lambda \mathbb{E}[D_f(s)] + 6 \left( \frac{k(k+1)}{2} \lambda \right)^2 \text{Var}[D_f(s)]}{6\sqrt{\text{Var}[Q_f^{(1)}(s)]} \sqrt{d_1 t^{2 \rho}}} \\=& \left( \frac{k(k+1)(2k+1)\lambda \mathbb{E}[D_f(s)] + 6 \left( \frac{k(k+1)}{2} \lambda \right)^2 \text{Var}[D_f(s)]}{6\sqrt{d_1 \text{Var}[Q_f^{(1)}(s)]}} \right) t^{-\rho},
\end{align*}
which decays like the power law $ t^{-\rho}, ~~0< \rho <1 $. Hence the TCPPoK-I exhibits the LRD property.
\end{proof}

\begin{lemma}
	The PPoK has the LRD property. 
\end{lemma}
\begin{proof}
	%{\color{red} Fill in some more details in the proof.}\\
	Let $ 0 \leq s <t < \infty $, then
	%$$ Cov(N^{(k)}(s),N^{(k)}(t)) = \frac{k(k+1)(2k+1)}{6}\lambda s,~~ Var[N^{(k)}(t)] = \frac{k(k+1)(2k+1)}{6}\lambda t. $$
	$$ Corr[N^{(k)}(s),N^{(k)}(t)] = %\frac{Cov(N^{(k)}(s),N^{(k)}(t))}{\sqrt{Var[N^{(k)}(s)]Var[N^{(k)}(t)]}} 
	 s^{\frac{1}{2}}t^{- \frac{1}{2}} $$
	$$ \Rightarrow \lim_{t \rightarrow \infty} \frac{Corr[N^{(k)}(s),N^{(k)}(t)]}{t^{-d}} = \lim_{t \rightarrow \infty} \frac{s^{\frac{1}{2}}t^{- \frac{1}{2}}}{t^{-\frac{1}{2}}} 
	% s^{\frac{1}{2}}
	 = c(s).$$
	From the Definition \ref{Def:LRD}, we can say that the PPoK has the LRD property.
\end{proof}

\subsection{Limit theorems} In this subsection, we derive some results on limit theorems of the PPoK and the TCPPoK-I.
\begin{lemma}
Let $\{N^{(k)}(t)\}_{t\geq 0}$ be the PPoK. Then
\begin{equation}\label{lemma:limit-theorem}
\lim\limits_{t\to\infty}\frac{N^{(k)}(t)}{t}=\frac{k(k+1)}{2}\lambda,~in~probability.
\end{equation}
\begin{proof}
We know that the PPoK can be represented as sum of $k$ independent Poisson processes 
$ N_1(t),N_2(t),\ldots,N_k(t)$ (see \cite{Poiss-order-k}).
% with means $\mathbb{E}[N_i(t)] = \lambda t$ and \textit{pgf} $ \psi_{N_i(t)}(s) = e^{-\lambda t(1-s^i)}, ~i=1,2,..,k $.\\
%That is, we can express the PPoK as follows
%So it can also be written as $ M_1(t)+2M_2(t)+...+kM_k(t)$
\begin{equation*}\label{def:repres}
N^{(k)}(t) \stackrel{d}{=} N_1(t) + 2 N_2(t) +3  N_3(t) +.....+k N_k(t).
\end{equation*}  
%where $ N_i(t)$'s,$~ i=1,2,...,k$ are independent Poisson process with intensity $\lambda>0$.
Consider
 \begin{align*}
\lim\limits_{t\to\infty}\frac{N^{(k)}(t)}{t} =& \lim\limits_{t\to\infty}\frac{N_1(t) + 2 N_2(t) +3  N_3(t) +.....+k N_k(t)}{t}, ~\text{in distribution}  \\ 
=& \lim\limits_{t\to\infty}\frac{N_1(t)}{t}+  2 \lim\limits_{t\to\infty}\frac{N_2(t)}{t} +\ldots+k \lim\limits_{t\to\infty}\frac{N_k(t)}{t}, ~\text{in distribution.} 
\shortintertext{Using the law of large numbers and as limit in distribution goes to a constant, we get}  =& \lambda +2\lambda + \ldots+ k\lambda,~ \text{in~probability}, ~\\
 =& \frac{k(k+1)}{2}\lambda,~\text{ in~probability}.\qedhere
\end{align*}
\end{proof}
\end{lemma}
\noindent Next, we prove limit theorem for TCPPoK-I. To do so, we first need the following definition.
\begin{definition} We call a function $ l:(0,\infty) \rightarrow (0,\infty)$ regularly varying at $0+$ with index $\alpha \in \mathbb{R}$ if 
	$$ \lim_{x\rightarrow 0+} \frac{l(\lambda x)}{l(x)} = \lambda^{\alpha},~  \lambda >0. $$
\end{definition}
\noindent   The following result of the law of iterated logarithm for subordinator is reproduced from \cite[Chapter III, Theorem 14]{bertoin}.
\begin{lemma}Let $D_f(t)$ be a subordinator with $\mathbb{E}[e^{-sD_f(t)}]=e^{-tf(s)}$, where $f(s)$ is regularly varying at $0+$ with index $\alpha\in(0,1)$. Let $h$ be the inverse function of $f$ and
	\begin{equation*}
	g(t)=\frac{\log\log t}{h(t^{-1}\log\log t)},~(e<t).
	\end{equation*}
	Then
	\begin{equation}\label{LIL-sub}
	\liminf_{t\to\infty}\frac{D_f(t)}{g(t)}=\alpha(1-\alpha)^{(1-\alpha)/\alpha},~a.s.
	\end{equation}
\end{lemma}
\begin{theorem}Let the Laplace exponent $f(s)$ of the subordinator $ D_f(t)$ be regularly varying at $ 0+$ with index $ \alpha \in (0,1)$. Then 
$$ \liminf_{t \rightarrow \infty} \frac{Q_f^{(1)}(t)}{g(t)} =  \frac{k(k+1)}{2}\lambda \alpha (1-\alpha)^{(1-\alpha)/(\alpha)},~ in~probability, $$
where
$$ g(t) = \frac{\log \log t }{f^{-1}(t^{-1}\log\log t)}~(e<t).$$
\end{theorem}
\begin{proof}
 We know that, by definition, $ Q_f^{(1)}(t) = N^{(k)}(D_f(t))$. Now,
\begin{align*}
 \liminf_{t \rightarrow \infty} \frac{Q_f^{(1)}(t)}{g(t)} =& \liminf_{t \rightarrow \infty} \frac{ N^{(k)}(D_f(t))}{g(t)} \\ =&
 \liminf_{t \rightarrow \infty} \frac{ N^{(k)}(D_f(t))}{D_f(t)}\frac{D_f(t)}{g(t)}
 \shortintertext{Note that $D_f(t)\to\infty$, $a.s.$ as $t\to\infty$ (see  \cite[Section 1.5.1]{appm}). We have that}
 =& \frac{k(k+1)}{2}\lambda \liminf_{t \rightarrow \infty} \frac{D_f(t)}{g(t)},~in~probability~ 
\text{(using \eqref{lemma:limit-theorem})} \\ =& \frac{k(k+1)}{2}\lambda \alpha (1-\alpha)^{(1-\alpha)/(\alpha)},~in~probability,~
 \end{align*}
where the last step follows from \eqref{LIL-sub}, which completes the proof.
\end{proof}
\section{Time changed Poisson process of order $k$-II }\label{sec:tcfppokII}
\noindent In this section, we consider the PPoK time-changed by inverse of L\'evy subordinator.\\
The first exit time of the subordinator $ D_f(t)$, called as inverse subordinator, is defined by
\begin{equation*}
 E_f(t) = \inf \{r\geq 0:D_f(r)> t \},~t\geq 0.
 \end{equation*}
\begin{definition}
The time-changed PPoK of Type-II (TCPPoK-II) is defined as
$$ Q_f^{(2)}(t) = N^{(k)}(E_f(t)),~t\geq 0,$$
where $ N^{(k)}(t)$ is independent of the inverse subordinator $ \{E_f(t)\}_{t\geq0}$.
\end{definition}
\noindent
As proved in the case of TCPPoK-I, one can prove the following results on similar lines.\\
The \textit{pmf} of  the TCPPoK-II is given by
$$  P[Q_f^{(2)}(t) = n] = \sum_{{\bf x} \in \Omega(k,n)} \frac{\lambda^{\zeta_k}}{\Pi_k!}\sum_{m=0}^{\infty} \frac{(-k\lambda)^m}{m!} \mathbb{E}[E_f^{\zeta_k+m}],~~n=0,1,2,\ldots. $$
Let $0<s\leq t<\infty$, then the mean and covariance function of TCPPoK-II are given by
\begin{enumerate}[(i)]
	\item $  \mathbb{E}[Q_f^{(2)}(t)]= \frac{k(k+1)}{2} \lambda \mathbb{E}[E_f(t)] $
	
	\item $ \mbox{Cov}[Q_f^{(2)}(s),Q_f^{(2)}(t)] = \frac{k(k+1)(2k+1)}{6}\lambda \mathbb{E}[E_f(s)] + \left(\frac{k(k+1)}{2} \lambda\right)^2 \text{Var}[E_f(s)].$
\end{enumerate}
\ifx
\begin{theorem}
The \textit{pmf} of the TCPPoK-II can be written as
$$  P[Q_f^{(2)}(t) = n] = \sum_{{\bf x} \in \Omega(k,n)} \frac{\lambda^{\zeta_k}}{\Pi_k!}\sum_{m=0}^{\infty} \frac{(-k\lambda)^m}{m!} \mathbb{E}[E_f^{\zeta_k+m}]. $$
%where the summation is over all nonnegative integers $ x_1,x_2,...,x_k $ such that 
%$$ \sum_{j=1}^{k}jx_j=n.$$ 
\begin{proof}
The proof runs similar to the proof of Theorem \ref{thm:tcppk-I-pmf} and hence is omitted here.
\end{proof}
\end{theorem}

\begin{theorem} Let $ 0<s \leq t < \infty$, the distributional properties of the TCPPoK-II are as follows
	\begin{enumerate}[(i)]
		\item $  \mathbb{E}[Q_f^{(2)}(t)]= \frac{k(k+1)}{2} \lambda \mathbb{E}[E_f(t)] $

		\item $ \mbox{Cov}[Q_f^{(2)}(s),Q_f^{(2)}(t)] = \frac{k(k+1)(2k+1)}{6}\lambda \mathbb{E}[E_f(s)] + \left(\frac{k(k+1)}{2} \lambda\right)^2 \text{Var}[E_f(s)].$
	\end{enumerate}
\begin{proof}
The proof is similar as we proved for the case of TCPPoK-I.
To get the expression of variance of the TCPPoK-II, we can put $s=t$ in the Part (ii).
\end{proof}	
\end{theorem}
\fi
\noindent Now, we discuss the asymptotic behavior of moments of the TCPPoK-II. First we need the following Tauberian theorem (see \cite{bertoin,Taqqu2010}).
\begin{theorem}(Tauberian Theorem) \label{thm:taub}
Let $ l:(0,\infty)\rightarrow (0,\infty)$ be a slowly varying function at $0 $ (respectively $\infty $) and let $\rho \geq 0$. Then for a function $ U:(0,\infty)\rightarrow (0,\infty) $, the following are equivalent
\begin{enumerate}[(i)]
\item $ U(x) \sim x^{\rho}l(x)/ \Gamma(1+\rho),~~~x\rightarrow 0 ~(respectively ~x\rightarrow \infty).$
\item $\tilde{U}(s) \sim s^{-\rho-1}l(1/s),~~~s \rightarrow \infty~ (respectively ~s\rightarrow 0), $ where $\tilde{U}(s) $ is the LT of $U(x).$
\end{enumerate}
\end{theorem}
\noindent The Laplace Transform (LT) of $p$th moment of $ E_f(t) $ is given by  (see \cite{TCFPP-pub})
$$ \tilde{M}(s) = \frac{\Gamma(1+p)}{s(f(s))^p},~p>0,$$
where $ f(s)$ is the corresponding Bernstein function associated with L\'evy subordinator $ D_f(t)$.
\begin{example}[PPoK time-changed by inverse gamma subordinator] Let $E_Y(t)$ be the first hitting time of gamma subordinator $Y(t)$ as defined in Remark \ref{rem:levysub}(i) is defined as
$$ E_Y(t) = \inf \{ r \geq 0:~Y(r)> t \},~~ t\geq 0.$$
\ifx	
 Let $\{Y(t)\}_{t\geq 0}$ be the gamma subordinator with corresponding Bernstein function $f(s) = p\log(1+\frac{s}{\alpha})$.  We consider the first hitting time of gamma subordinator, which is called as the inverse gamma subordinator.  It is defined as
$$ E_Y(t) = \inf \{ r \geq 0:~Y(r)> t \},~~ t\geq 0.$$
\fi
We study the asymptotic behavior of mean of  the TCPPoK-II $\{Q_Y^{(2)}(t)\}_{t\geq 0}$. 
The LT of $\mathbb{E}[E_Y(t)]$ is given by
$$ \tilde{M_Y}(s) = \frac{\Gamma(2)}{s(p\log(1+\frac{s}{\alpha}))}.$$
It can be seen that $$p\log\left(1+\frac{s}{\alpha}\right) \sim \frac{ps}{\alpha},~s \rightarrow 0 \Rightarrow \tilde{M_Y}(s) \sim \dfrac{\Gamma(2)s^{-2}\alpha}{p},~s\to 0. $$ 
Then by Theorem \ref{thm:taub}, we have that 
$$ \mathbb{E}[Q^{(2)}_Y(t)] = \frac{k(k+1)}{2}\lambda \mathbb{E}[E_Y(t)] \sim \frac{k(k+1)}{2}\lambda \frac{t\alpha}{p},~ as ~t \rightarrow \infty.$$
In a similar manner, we can compute the asymptotic behavior of $ \text{Var}[Q^{(2)}_Y(t)]$.
 \begin{align*}
\text{Var}[Q^{(2)}_Y(t)] =&
%\frac{k(k+1)(2k+1)}{6}\lambda \mathbb{E}[E_Y(t)] + \left(\frac{k(k+1)}{2} \lambda\right)^2 \text{Var}[E_Y(t)] \\
\frac{k(k+1)(2k+1)}{6}\lambda \mathbb{E}[E_Y(t)] + \left(\frac{k(k+1)}{2} \lambda\right)^2 [\mathbb{E}[E_Y(t)^2]-\mathbb{E}[E_Y(t)]^2 ] \\
\sim & \frac{k(k+1)(2k+1)}{6}\lambda \left(\frac{t\alpha}{p}\right)+ \left(\frac{k(k+1)}{2} \lambda\right)^2 \left[\left(\frac{t\alpha}{p}\right)^2 - \left(\frac{t\alpha}{p}\right)^2 \right] ,~~as~t\rightarrow \infty \\ \sim & \frac{k(k+1)(2k+1)}{6}\lambda \left(\frac{t\alpha}{p}\right),~~as~t\rightarrow \infty. 
\end{align*}
 
\end{example}
\begin{example}[PPoK time-changed by the inverse tempered $\alpha$-stable subordinator] We consider the PPoK time-changed by the inverse tempered $\alpha$-stable subordinator $ E_{\alpha}^{\mu}(t)$ (see \cite{Kumar-ITSS}). The asymptotic behavior of the $p$-th moment of $E_{\alpha}^{\mu}(t)$ is given by (see \cite[Proposition 3.1]{Kumar-ITSS})
\begin{equation*}
 \mathbb{E}[(E_{\alpha}^{\mu}(t))^p] \sim \begin{cases} \dfrac{\Gamma(1+p)}{\Gamma(1+p\alpha)}t^{p\alpha},&~~as~t \rightarrow 0, \\&\\ \dfrac{\lambda^{p(1-\alpha)}}{\alpha^p}t^p,&~as~t\rightarrow \infty.
 \end{cases}
\end{equation*}
We consider the case for $ p=1$, then by Theorem \ref{thm:taub}, we get
\begin{equation*}
 \mathbb{E}[Q^{(2)}_{\mu,\alpha}(t)] = \frac{k(k+1)}{2}\lambda \mathbb{E}[E_{\alpha}^{\mu}(t)] \sim \begin{cases} \dfrac{k(k+1)\lambda\Gamma(2)}{2 \Gamma(1+\alpha)}t^{\alpha},&~as~t \rightarrow 0, \\&\\ \dfrac{k(k+1)\lambda^{(2-\alpha)}}{2 \alpha}t,&~as~t\rightarrow \infty.
 \end{cases}
\end{equation*}
\end{example}
\begin{example}[PPoK time-changed with inverse of the inverse Gaussian subordinator] Let $ E_{G}(t)$ be the right-continuous inverse of the inverse Gaussian  subordinator $\{G(t)\}_{t\geq 0}$ as defined in Remark \ref{rem:levysub}(iii). It is defined as
	 \begin{equation*}
	E_G(t)=\inf\{r\geq 0:G(r)>t \}, t\geq 0.	
	\end{equation*} The mean of $E_{G}(t)$  is given by (see \cite{Kumar-ITSS,TCFPP-pub})
	$$ M(t)=\mathbb{E}[E_{G}(t)] = \frac{\Gamma(2)}{s(\delta(\sqrt{2s+\gamma^2} - \gamma))}.$$

\noindent Taking the LT of the above expression, we get
\begin{equation*}
 \tilde{M}(s) \sim \begin{cases} \dfrac{\Gamma (2)}{(\delta / \gamma)}s^{-2},&~as~ s \rightarrow 0, \\&\\ \dfrac{\Gamma (2)}{(\delta \sqrt{2})}s^{-\frac{3}{2}},&~as~ s \rightarrow \infty.
 \end{cases}  
\end{equation*}
Using Theorem \ref{thm:taub}, we have that
\begin{equation*}
\mathbb{E}[Q^{(2)}_G(t)] = \frac{k(k+1)}{2}\lambda \mathbb{E}[E_{G}(t)] \sim \begin{cases} \dfrac{k(k+1)\lambda\Gamma(2)}{2 \Gamma(1+\frac{1}{2})(\delta \sqrt{2})}t^{\frac{1}{2}},&~as~t \rightarrow 0, \\&\\ \dfrac{k(k+1)}{2}\lambda (\frac{\gamma}{\delta})t,&~as~t\rightarrow \infty.
 \end{cases}
\end{equation*}

\end{example}
%\begin{lemma} (see \cite{TCFPP-pub})
%Let $ E_{f_1}(t)$ and $ E_{f_2}(t)$ be two independent inverse subordinators corresponding to Bernstein function $ f_1(s)$ and $ f_2(s)$ respectively . Then 
%$$ \{ E_{f_1}(E_{f_2}(t))\} = \{ E_{f_1of_2}(t)\},~ t \geq 0 $$
%where $ (f_1o f_2)(s) = f_1(f_2(s))$.
%\end{lemma} 

%\newpage

\section{Governing equation for time-changed Poisson processes of order $k$}\label{sec:dde}

\noindent
Stochastic processes are intimately connected with partial differential equations (\textit{pde}) (e.g. Brownian motion and its diffusion equation), and difference-differential equation (\textit{dde}) (Poisson process and its governing equation). In this section, we present the governing equations for some special cases of the TCPPoK-I and the TCPPoK-II.
\subsection{Governing equation for Poisson-inverse Gaussian process of order $k$}
Let $N^{(k)}(t)$ be the PPoK and $G(t) \sim IG(\delta t,\gamma)$ be the inverse Gaussian  subordinator. Then density function $g(x,t)$ of $G(t)$ solves the following \textit{pde} (see \cite{Kumar-TCPP})
\begin{equation*}
 \frac{\partial^2}{\partial t^2}g(x,t)-2\delta \gamma \frac{\partial}{\partial t}g(x,t) = 2\delta^2\frac{\partial}{\partial x}g(x,t).
 \end{equation*}

\noindent
We derive the governing equation for the TCPPoK-I.%, when the time-change is done by inverse Gaussian subordinator. 
\begin{theorem}
Let $ \hat{p}_m(t)$ denote the {\textit{pmf}} of the TCPPoK-I $ \{N^{(k)}(G(t))\}_{t\geq 0}$. Then it solves the following \textit{dde}
\begin{equation*}
 \left(\frac{d^2}{dt^2}-2\delta \gamma \frac{d}{dt}\right) \hat{p}_m(t) = 2\delta^2 \lambda \left[k\hat{p}_m(t)-\left(\hat{p}_{m-1}(t)+\hat{p}_{m-2}(t)+...+\hat{p}_{m-m\wedge k}(t)\right)\right]
\end{equation*}
\end{theorem}
\begin{proof}
We know that 
 $$ \hat{p}_m(t) = \mathbb{P}[N^{(k)}(G(t)) = m] = \int_0^{\infty} p_m(x)g(x,t)dx. $$
%This implies by the dominated convergence theorem \textcolor{red}{(type reason here in detail why DCT applies here)}\\
Since $g(x,t)$ is measurable and integrable, we have the following expression
$$ \frac{d}{dt} \hat{p}_m(t) = \int_0^{\infty} p_m(x) \frac{\partial}{\partial t} g(x,t)dx,$$
and
$$ \frac{d^2}{dt^2} \hat{p}_m(t) = \int_0^{\infty} p_m(x) \frac{\partial^2}{\partial t^2} g(x,t)dx. $$
Consider now
\begin{align*}
\left(\frac{d^2}{dt^2}-2\delta \gamma \frac{d}{dt}\right) \hat{p}_m(t) =& \int_0^{\infty}p_m(x) \left(\frac{\partial^2}{\partial t^2}-2\delta \gamma \frac{\partial}{\partial t}\right) g(x,t)dx \\ =& 2\delta^2 \int_0^{\infty} p_m(x)\frac{\partial}{\partial x} g(x,t)dx
\intertext{On applying integration by parts and using ~$\lim_{x\rightarrow \infty}g(x,t)=\lim_{x\rightarrow 0}g(x,t)=0,$ we get} =& -2\delta^2 \int_0^{\infty} \frac{d}{dx} p_m(x) g(x,t)dx \\ =& -2\delta^2 \int_0^{\infty} [-k\lambda p_m(x)+\lambda [p_{m-1}(x)+p_{m-2}(x)+\ldots+p_{m-m\wedge k}(x)]g(x,t)dx \\ 
%=& -2\delta^2 \lambda [-k\hat{p}_m(t)+\hat{p}_{m-1}(t)+\hat{p}_{m-2}(t)+\ldots+\hat{p}_{m-m\wedge k}(t)]\\ 
=&  2\delta^2 \lambda \left[k\hat{p}_m(t)-\left(\hat{p}_{m-1}(t)+\hat{p}_{m-2}(t)+\ldots+\hat{p}_{m-m\wedge k}(t)\right)\right].\qedhere
\end{align*}
\end{proof}
\subsection{Governing equation for PPoK time-changed by hitting time of inverse Gaussian subordinator}
Next we consider the TCPPoK-II where the time-change is done by the hitting time of the inverse Gaussian process $G(t)$. The first hitting time of the process $G(t)$ is defined by
$$ E_G(t) = \inf \{ s \geq 0 : G(s)>t\} .$$
We know that (see \cite{Kumar-TCPP}) the density function $h(x,t)$ of $E_G(t)$ satisfies the following \textit{pde} 
$$ \frac{\partial^2}{\partial x^2}h(x,t)-2\delta \gamma \frac{\partial}{\partial x}h(x,t)= 2\delta^2 \frac{\partial}{\partial t}h(x,t)+ 2\delta^2h(x,0)\delta_0(t).  $$

\noindent To derive the governing \textit{dde} for the TCPPoK-II we first need \textit{dde} of PPoK for $K=2$. Keeping this in mind, we differentiate equation \eqref{PPoK-DE2} with respect to $t$, we get for $m=1,2,\ldots$
\begin{align}
\frac{d^2}{dt^2}p_m(t)&=\frac{d}{dt}\left(-k\lambda p_m(t)+\lambda\sum_{j=1}^{m\wedge k}p_{m-j}(t)\right),\nonumber\\
&=-k\lambda \frac{d}{dt}p_m(t)+\lambda\sum_{j=1}^{m\wedge k}\frac{d}{dt}p_{m-j}(t),\nonumber\\
&=-k\lambda \left(-k\lambda p_m(t)+\lambda\sum_{j=1}^{m\wedge k}p_{m-j}(t)\right)+\lambda\sum_{j=1}^{m\wedge k}\left(-k\lambda p_{m-j}(t)+\lambda\sum_{i=1}^{(m-j)\wedge k}p_{m-j-i}(t)\right)\nonumber\\
\frac{d^2}{dt^2}p_m(t)&=(k\lambda)^2p_m(t)-2k\lambda^2\sum_{j=1}^{m\wedge k}p_{m-j}(t)+\lambda^2\sum_{j=1}^{m\wedge k}\left(\sum_{i=1}^{(m-j)\wedge k}p_{m-j-i}(t)\right)\label{pde:double}
\end{align}
\begin{theorem}
Let the {\textit{pmf}} of the TCPPoK-II be denoted by $\hat{p}_m(t)=P[N^{(k)}(E_G(t))=m ]$. Then it satisfies the following \textit{dde} 
\begin{align*}
\frac{d}{dt} \hat{p}_m(t) &= \frac{1}{2\delta^2}\left[\int_0^{\infty} \left( (k\lambda)^2p_m(x)-2k\lambda^2\sum_{j=1}^{m\wedge k}p_{m-j}(x)\right.\right.\\&\left.\left.+\lambda^2\sum_{j=1}^{m\wedge k}\left(\sum_{i=1}^{(m-j)\wedge k}p_{m-j-i}(x)\right)\right)h(x,t)dx \right.\\&\left.+ 2\delta \gamma \int_0^{\infty} \left(-k\lambda p_m(x)+\lambda\sum_{j=1}^{m\wedge k}p_{m-j}(x)\right) h(x,t)dx + h(0,t) p_m'(0)\right]-\delta_0(t)\hat{p}_m(0),
\intertext{when $m=1,2,\ldots$ and }
\frac{d}{dt} \hat{p}_0(t) &=\frac{1}{2\delta^2}\left[\int_0^{\infty}  (k\lambda)^2p_0(x)h(x,t)dx -2k\lambda\delta \gamma \int_0^{\infty} p_0(x)h(x,t)dx -k\lambda h(0,t) \right]-\delta_0(t)\hat{p}_0(0), 
\end{align*}
when $m=0$ with initial condition
$ p_m'(0) = \begin{cases} -k\lambda&m=0,\\ \lambda&m=1,2,\ldots,k, \\ 0& m \geq k+1.
\end{cases}$
\end{theorem}
\begin{proof}
We first take the case when $m=1,2,\ldots.$ Consider 
\begin{align}
\frac{d}{dt} \hat{p}_m(t) =& \int_0^{\infty} p_m(x) \frac{\partial}{\partial t} h(x,t)dx,\nonumber \\ 
=& \frac{1}{2\delta^2} \int_0^{\infty} p_m(x)\left[\frac{\partial^2}{\partial x^2}h(x,t)-2\delta \gamma \frac{\partial}{\partial x}h(x,t)- 2\delta^2h(x,0)\delta_0(t)\right] dx,\nonumber \\ 
=& \frac{1}{2\delta^2} \int_0^{\infty} p_m(x)\left[\frac{\partial^2}{\partial x^2}h(x,t)-2\delta \gamma \frac{\partial}{\partial x}h(x,t)\right]dx-\delta_0(t)\int_0^{\infty} p_m(x)h(x,0)dx. \label{pde:tcp}
\end{align}
We will now consider the first term in the above equation
\begin{align*}
\int_0^{\infty} p_m(x)\frac{\partial^2}{\partial x^2}h(x,t)dx =& p_m(x)\frac{\partial}{\partial x}h(x,t)|_0^{\infty} - \int_0^{\infty} \frac{d}{dx}p_m(x)\frac{\partial}{\partial x}h(x,t)dx \\ =&  p_m(x)\frac{\partial}{\partial x}h(x,t)|_0^{\infty} - h(x,t)\frac{d}{dx}p_m(x)|_0^{\infty} + \int_0^{\infty} \frac{d^2}{dx^2}p_m(x)h(x,t)dx. \\ 
\intertext{Since $\lim_{x\rightarrow \infty}h_x(x,t)=\lim_{x\rightarrow \infty}h(x,t)=0 $ and $h_x(0,t) = 2\delta\gamma h(0,t),$ we get} 
  =& -2\delta \gamma p_m(0)h(0,t) + h(0,t) \frac{d}{dx}p_m(0)+ \int_0^{\infty} \frac{d^2}{dx^2}p_m(x)h(x,t)dx.
\intertext{Also,} 
\int_0^{\infty}p_m(x)\frac{\partial}{\partial x}h(x,t)dx =& p_m(x)h(x,t)|_0^{\infty}-\int_0^{\infty}\frac{d}{dx}p_m(x)h(x,t)dx\\ =& -p_m(0)h(0,t)-\int_0^{\infty}\frac{d}{dx}p_m(x)h(x,t)dx.
\end{align*}
Then Equation \eqref{pde:tcp} becomes
\begin{align*}
\frac{d}{dt} \hat{p}_m(t) =&
%\frac{1}{2\delta^2} \int_0^{\infty} p_m(x)\left[\frac{\partial^2}{\partial x^2}h(x,t)-2\delta \gamma \frac{\partial}{\partial x}h(x,t)\right]dx-\delta_0(t)\int_0^{\infty} p_m(x)h(x,0)dx \\
 \frac{1}{2\delta^2} \left[-2\delta \gamma p_m(0)h(0,t) + h(0,t) p_m'(0)+ \int_0^{\infty} p_m''(x)h(x,t)dx \right.\\ &\left.  - 2 \delta \gamma \{ -p_m(0)h(0,t)-\int_0^{\infty}p_m'(x)h(x,t)dx \}\right]-\delta_0(t)\hat{p}_m(0) \\ =& \frac{1}{2\delta^2}\left[\int_0^{\infty} p_m''(x)h(x,t)dx + 2\delta \gamma \int_0^{\infty}p_m'(x)h(x,t)dx + h(0,t) p_m'(0)\right]-\delta_0(t)\hat{p}_m(0).
\end{align*}
Substituting the expressions of of $p_m'(x)$ and $ p_m''(x)$ from \eqref{PPoK-DE2} and \eqref{pde:double}, respectively, we get the desired result.\\
%Similarly, expression for the case $m=0$ can be derived.
For $m=0$, $ p_0(t) = e^{-k\lambda t}$.\\
\begin{align*}
\frac{d}{dt} \hat{p}_0(t) =& \int_0^{\infty} p_0(x) \frac{\partial}{\partial t} h(x,t)dx \\ 
=& \frac{1}{2\delta^2} \int_0^{\infty} p_0(x)\left[\frac{\partial^2}{\partial x^2}h(x,t)-2\delta \gamma \frac{\partial}{\partial x}h(x,t)- 2\delta^2h(x,0)\delta_0(t)\right] dx\\
%\intertext{after solving we get}
=& \frac{1}{2\delta^2} \left[-2\delta \gamma p_0(0)h(0,t) + h(0,t) p_0'(0)+ \int_0^{\infty} p_0''(x)h(x,t)dx \right.\\ &\left.  - 2 \delta \gamma \{ -p_0(0)h(0,t)-\int_0^{\infty}p_0'(x)h(x,t)dx \}\right]-\delta_0(t)\hat{p}_0(0)\\
%=&\frac{1}{2\delta^2} \left[-2\delta \gamma h(0,t) -k\lambda h(0,t)+ \int_0^{\infty} (k\lambda)^2 p_0(x)h(x,t)dx \right.\\ &\left.  - 2 \delta \gamma \{ -h(0,t)-\int_0^{\infty}(-k\lambda)p_0(x)h(x,t)dx \}\right]-\delta_0(t)\hat{p}_0(0)\\
=& \frac{1}{2\delta^2}\left[\int_0^{\infty} (k\lambda)^2 p_0(x)h(x,t)dx-2k\lambda\delta\gamma \int_0^{\infty}p_0(x)h(x,t)dx-k\lambda h(0,t) \right] -\delta_0(t)\hat{p}_0(0)
\end{align*}
\end{proof}
\section{Application in Risk Theory}\label{sec:appl}
The classical insurance risk model is defined by 
\begin{equation*}
Z(t)=ct-\sum_{j=1}^{N(t)}Z_j, t\geq0,
\end{equation*}
where $\{N(t)\}_{t\geq 0}$ is the homogeneous Poisson process, which counts the number of claim arrivals upto time $t$ and $Z_j$ is the claim amount size with distribution $F$, independent of $N(t)$. The risk process models the cash flow of an insurance company where the premium rate is fixed at $c>0$. Though this models is simple and easy to use, but it does not cover all practical aspects of insurance ruin. In this section, we attempt to improve this model in following ways, namely,
\begin{enumerate}
	\item \textbf{Group insurance schemes}: Insurance companies sell group insurance policies for families, businesses and institutions, and \textit{etc.} where a single claim reporting implies several claims within a group. These situations can be modelled using PPoK (see \cite{Poiss-order-k}), where the claims arrive in groups of size less than or equal to $k$.
	\item \textbf{Ruin due to sudden large scale extreme events}: The classical Poisson process, as evident from its transition probability function, assigns extremely low probability to more than one event in a small time period. However, in practice, we have observed that natural and man-made calamities can force large number of claim arrivals in a short span of time. For example, after 9/11 attacks, the insurance companies were badly  affected by large scale claim arrivals in small time period. The Poisson process time-changed by L\'evy subordinator allows arbitrary arrivals in short span of time (see \cite{OrsToa-Berns,Orsingher2012}). 
\end{enumerate}
The model we proposed in this paper encapsulates the above improvements. Our proposed model reduces to group insurance scheme model when no time-change is done. It also covers sudden large scale extreme events when $k=1$ (in case of non group insurance schemes). In this section, we study ruin probability, joint distribution of time to ruin and deficit at ruin, and derive their governing equation based on our generalized model given below.

%Collective risk theory, as a part of insurance- or actuarial- mathematics, deals with stochastic models of an insurance business. In such a model the occurrence of the claims is described by a point process and the amounts of money to be paid by the company at each claim by a sequence of random variables.The company recieves a certain amount of premium to cover its liability. The company is furthermore assumed to have a certain initial capital $u$ at its disposal. One important problem in collective risk theory is to investigate the \textbf{ruin probability}.\\
%In classical risk model it is assumed that the arrival of claim follows Poisson distribution. We will generalize the classical risk model by taking the occurance of the claim is described by TCPPoK-I, and analyse its ruin probability. \\   

Let $ \{Q_f^{(1)}(t)\}_{t\geq 0}$ be the TCPPoK-I. Consider the risk model governed by the TCPPoK-I, denoted by $\{ X(t)\}_{t \geq 0}$, defined as
\begin{equation}\label{riskmodel}
 X(t)= ct-\sum_{j=1}^{Q_f^{(1)}(t)}Z_j,~~t\geq 0,
 \end{equation}
where $c>0$ denotes premium rate, which is assumed to be constant and $ Z_i$ be non-negative IID random variables with distribution $F$, representing the claim size.
The ratio of $\mathbb{E}[X(t)]$ and $ \mathbb{E}[\sum_{j=1}^{Q_f^{(1)}(t)} Z_j]$ is called premium loading factor, denoted by $\rho$, is given by
$$ \rho = \frac{ \mathbb{E}[X(t)]}{\mathbb{E}[\sum_{j=1}^{Q_f^{(1)}(t)} Z_j]} = \frac{ct}{\mu \mathbb{E}[Q_f^{(1)}(t)]}-1,$$
where $\mu=\mathbb{E}[Z_j]$. The premium loading factor signifies the profit margin of the insurance firm. %Hence large value of $\rho$ means that the income flow (determined by $c$) exceeds outgoing flow (determined by $\mu$ and $ \mathbb{E}[Q_f^{(1)}(t)]$ ).\\
%Normally $\rho >0$-means company is profitable in average and one can make $\rho <0$ by reducing gross premium rate $c$ in order to win new customers.\\
Let us denote the initial capital by $u>0$. Define the surplus process  $\{U(t)\}_{t\geq 0}$ by
$$ U(t) = u+X(t),~~t\geq 0.$$ 
The insurance company will be called in ruin if the surplus process falls below zero level. Let $T$ denote the first time to ruin and is defined as
$$ T=\inf\{t>0: U(t)<0\}.$$
Then probability of ruin is given by
$$ \psi(u)=\mathbb{P}\{T<\infty\}.$$
The joint probability that ruin happens in finite time and the deficit at the time of ruin, which is denoted as $ D=|U(t)|$, is given by
%Let $G(u,y)$ be the joint probability distribution of the time to ruin and the deficit at the time of ruin $ D_u=|U(t)|=|u+X(t)|$, that is,
\begin{equation}\label{jointpdf}
G(u,y)= \mathbb{P}\{T<\infty, D\leq y \},~~y\geq 0.
\end{equation}
Observe that
$$ \psi(u)=\lim_{y\rightarrow \infty} G(u,y).$$
Denote $u':=u+ch$. Now, using \eqref{tran-prob}, we get
\begin{align*}
G(u,y) =& (1-hf(k\lambda))G(u',y) -h\sum_{{\bf x} \in \Omega(k,\mathbf{1})}\frac{ (- \lambda )^{\zeta_k}}{\Pi_k!}f^{(\zeta_k)}(k\lambda)  \left[\int_0^{u'}G(u'-x,y)dF(x) + \right. \\& \left.  F(u'+y)-F(u') +\ldots +\int_0^{u'}G(u'-x,y)dF^{*k}(x) +  F^{*k}(u'+y)-F^{*k}(u')\right] \\
& -h \sum_{{\bf x} \in \Omega(k,\mathbf{2})} \frac{ (- \lambda )^{\zeta_k}}{\Pi_k!}f^{(\zeta_k)}(k\lambda) \left[\int_0^{u'}G(u'-x,y)dF(x)\right.  + F(u'+y)-F(u') +\ldots\\& \left.+ \int_0^{u'}G(u'-x,y)dF^{*k}(x) +   F^{*k}(u'+y)-F^{*k}(u')\right]\\& + \ldots \\&\vdots\hspace*{2cm}\vdots\hspace*{2cm}\vdots\hspace*{2cm}\vdots\hspace*{2cm}\vdots \\
=& (1-hf(k\lambda))G(u',y) \\&-h\sum_{n=1}^{\infty} \sum_{{\bf x} \in \Omega(k,n)}\frac{ (- \lambda )^{\zeta_k}}{\Pi_k!}f^{(\zeta_k)}(k\lambda) \left[ \sum_{i=1}^k \int_0^{u'}G(u'-x,y)dF^{*i}(x) + F^{*i}(u'+y)-F^{*i}(u') \right].
\end{align*}
After rearranging the terms, we have that
\begin{align*}
G(u',y)-G(u,y) =& hf(k\lambda)G(u',y)+ \\
& h\sum_{n=1}^{\infty} \sum_{{\bf x} \in	\Omega(k,n)}\frac{(-\lambda)^{\zeta_k}}{\Pi_k!}f^{(\zeta_k)}(k\lambda) \left[ \sum_{i=1}^k \int_0^{u'}G(u'-x,y)dF^{*i}(x) + F^{*i}(u'+y)-F^{*i}(u') \right]\\
\frac{G(u',y)-G(u,y)}{ch}
=& \frac{1}{c} f(k\lambda)G(u',y)+\\&  \frac{1}{c} \sum_{n=1}^{\infty} \sum_{{\bf x} \in \Omega(k,n)}\frac{ (- \lambda )^{\zeta_k}}{\Pi_k!}f^{(\zeta_k)}(k\lambda)  \left[ \sum_{i=1}^k \int_0^{u'}G(u'-x,y)dF^{*i}(x) + F^{*i}(u'+y)-F^{*i}(u') \right].
\end{align*}
Now taking $ h\rightarrow 0$, we get
\begin{align*}
\frac{\partial G}{\partial u} = &\frac{f(k\lambda)}{c}G(u,y)+\frac{1}{c} \sum_{n=1}^{\infty} \sum_{{\bf x} \in \Omega(k,n)}\frac{ (- \lambda )^{\zeta_k}}{\Pi_k!}f^{(\zeta_k)}(k\lambda)  \left[ \sum_{i=1}^k \int_0^{u}G(u-x,y)dF^{*i}(x) + F^{*i}(u+y)-F^{*i}(u) \right].
\shortintertext{The term in bracket in above expression, $F^{*i}$, represents the $i$-fold convolution of the claim size distribution. Let us denote the aggregate claims by  $B(x) = \sum_{i=1}^k F^{*i}(x)$  and normalizing it to a probability distribution by defining $B_1(x)= \frac{B(x)}{k}$, then we have}
\frac{\partial G}{\partial u} =&\frac{f(k\lambda)}{c}G(u,y)+ \frac{k}{c} \left[ \int_0^{u}G(u-x,y)dB_1(x) + B_1(u+y)-B_1(u)\right] \sum_{n=1}^{\infty} \sum_{{\bf x} \in \Omega(k,n)}\frac{ (- \lambda )^{\zeta_k}}{\Pi_k!}f^{(\zeta_k)}(k\lambda).
\end{align*}
Consider the last term in above expression, we obtain
\begin{align*}
\sum_{n=1}^{\infty} \sum_{{\bf x} \in \Omega(k,n)}\frac{ (- \lambda )^{\zeta_k}}{\Pi_k!}f^{(\zeta_k)}(k\lambda)=& \sum_{n=1}^{\infty} \sum_{\substack{x_1,x_2,\ldots x_k \geq 0\\ x_1+2x_2+\ldots+kx_k=n}}\frac{ (- \lambda )^{x_1+x_2+\ldots+x_k}}{x_1!x_2!\dots x_k!}f^{(x_1+x_2+\ldots+x_k)}(k\lambda).
\shortintertext{
Set  $x_i = n_i,  i=1,2,\ldots k$ and $n=x+\sum_{i=1}^k(i-1)n_i$. We get}
=& \sum_{x=1}^{\infty} \sum_{\substack{n_1,n_2,\ldots n_k \geq 0\\ n_1+n_2+\ldots+n_k=x}}\frac{ (- \lambda )^{n_1+n_2+\ldots+n_k}}{n_1!n_2!\dots n_k!}f^{(n_1+n_2+\ldots+n_k)}(k\lambda)
\\=& \sum_{x=1}^{\infty} \frac{(-\lambda)^x}{x!} f^{(x)}(k\lambda) \sum_{\substack{n_1,n_2,\ldots n_k \geq 0\\ n_1+n_2+\ldots+n_k=x}}\frac{(n_1+n_2+\ldots+n_k)!}{n_1!n_2!\dots n_k!}
 \\=& \sum_{x=1}^{\infty}\frac{(-\lambda)^x}{x!}f^{(x)}(k\lambda)(1+1+\ldots+1)^x =\sum_{x=1}^{\infty}\frac{(-\lambda k)^x}{x!}f^{(x)}(k\lambda) \\ =& \sum_{x=0}^{\infty}\frac{(-\lambda k)^x}{x!} f^{(x)}(k\lambda) -f(k\lambda).
\shortintertext{As $f$ is Bernstein function, it is infinitely differentiable and using Taylor's series, we get} =& f(k\lambda - k\lambda)-f(k\lambda)= -f(k\lambda) ~(\text{using }f(0)=0).
\end{align*}
%\begin{align*}
%\sum_{n=1}^{\infty}(-1)^n \frac{\lambda^n}{n!}f^{(n)}(\lambda)=& -\lambda f^{(1)}+ \frac{\lambda^2}{2!}f^{(2)}- \ldots \\ =& f(\lambda)-\lambda f^{(1)}+ \frac{\lambda^2}{2!}f^{(2)}- \ldots -f(\lambda)\\ =& f(\lambda-\lambda)-f(\lambda)\\ =& f(0)-f(\lambda)\\=& -f(\lambda)~~~(f(0)=0)
%\end{align*} 
\noindent From above calculations, we have the following result.
\begin{theorem} Let $G(u,y)$, defined in \eqref{jointpdf}, denote the joint probability distribution of time to ruin and deficit at this time of the risk model \eqref{riskmodel}. Then, it satisfies the following differential equation
\begin{equation}\label{appl:thmde}
\frac{\partial G(u,y)}{\partial u} = \frac{f(k\lambda)}{c}\left[ G(u,y)-k \left(\int_0^{u}G(u-x,y)dB_1(x) + B_1(u+y)-B_1(u)\right) \right].
\end{equation}
\end{theorem}
\begin{theorem}
The joint distribution of ruin time and deficit at ruin when the initial capital is zero, $G(0,y)$, is given by
\begin{equation}\label{appl:thm2}
 G(0,y)= \frac{f(k\lambda)}{c}\left[(k-1)\int_0^{\infty}G(u,y)du+ k \int_0^{\infty}[B_1(u+y)-B_1(u)]du\right].\end{equation}
\end{theorem}
\begin{proof}
On integrating \eqref{appl:thmde} with respect to $u$ on $ (0,\infty)$, we get
\begin{align*}
 G(\infty,y)-G(0,y)&= \frac{f(k\lambda)}{c} \left[\int_0^{\infty}G(u,y)du- k \left( \int_0^{\infty} \int_0^{u}G(u-x,y)dB_1(x)du \right.\right.+\\
& \left.\left. \int_0^{\infty}[B_1(u+y)-B_1(u)]du \right) \right].
 \end{align*}
Note that $G(\infty,y)=0$,then
% we have that
%$$ -G(0,y)= \frac{f(k\lambda)}{c} \left[\int_0^{\infty}G(u,y)du- k\left( \int_0^{\infty}G(u,y)du + \int_0^{\infty}[B_1(u+y)-B_1(u)]du\right) \right]. $$
%After rearranging the terms, we observe that
$$  G(0,y)= \frac{f(k\lambda)}{c}\left[(k-1)\int_0^{\infty}G(u,y)du+ k \int_0^{\infty}[B_1(u+y)-B_1(u)]du\right].$$\qedhere
\end{proof}
\begin{remark}
On taking limit $y\rightarrow \infty $ in \eqref{appl:thm2}, we get
$$ \psi(0) = \frac{f(k\lambda)}{c}\left[(k-1)\int_0^{\infty}\psi(u)du+ k \int_0^{\infty}[1-B_1(u)]du\right].$$
\end{remark}
\begin{remark}
From \eqref{appl:thmde}, we have that
$$\frac{\partial G}{\partial u} = \frac{f(k\lambda)}{c}\left[ G(u,y)- k\left(\int_0^{u}G(u-x,y)dB_1(x) + B_1(u+y)-B_1(u)\right) \right].$$
As $\lim_{y\rightarrow \infty} G(u,y)= \psi(u)$, on taking limit as $ y \rightarrow \infty $ in the above equation, we obtain the  following differential equation governing the ruin probability
$$ \frac{\partial \psi}{\partial u} = \frac{f(k\lambda)}{c}\left[ \psi(u)- k\left(\int_0^{u}\psi(u-x)dB_1(x) + (1-B_1(u))\right) \right].$$
\end{remark}

\section{Simulation}\label{sec:simu}
\noindent In this section, we present the algorithm to generate simulated sample paths for some TCPPoK-I and TCPPoK-II processes. Using the algorithms presented here, we generate simulated sample paths for the PPoK, the TCPPoK-I subordinated with gamma and inverse Gaussian subordinator, and the TCPPoK-II subordinated with  inverse gamma  and inverse of inverse Gaussian subordinator for a chosen set of parameters. We first present the algorithm for simulation of sample paths of the PPoK.
\begin{algorithm}[\bf Simulation of the PPoK]\label{simulation-fpp}%\vspace*{0.2cm}
	This algorithm (see \cite{Cahoy2010}) gives the number of events $N^{(k)}(t),t\geq 0$ of the PPoK up to a fixed time $T$. 
	\begin{enumerate}[(a)]
		\item Fix the parameters $\lambda>0$ and $k\geq1$ for the PPoK process.
		\item	Set $n=0, a=0$ and $t=0.$
		\item Repeat while $t<T$
		\begin{enumerate}
			\item[] Generate a uniform random variables $U$.
			\item[] Compute $t\leftarrow t+\left[-\frac{1}{\lambda}\ln U\right]$.
			\item[] Generate an independent random variable $X$ with discrete uniform distribution on $k$ points.
			\item[]  $a\leftarrow a+X$ and $n\leftarrow n+1$.
			
		\end{enumerate}
		\item Next $t$.
	\end{enumerate}
	Then $n$ denotes the number of events $N^{(k)}(t)$ occurred up to time $T$.
\end{algorithm} 
\noindent We next present a general algorithm to simulate the TCPPoK-I, subordinated with gamma subordinator  and the inverse Gaussian subordinator.
The same algorithm can be used to simulate the TCPPoK-II, subordinated with inverse gamma and inverse of inverse Gaussian processes. We refer to Algorithm 2--5 from \cite{TCFPP-pub} to generate sample paths of the gamma and the inverse Gaussian subordinator and their right-continuous inverses.
\begin{algorithm}[\bf  Simulation of the TCPPoK-I and the TCPPoK-II]
	\begin{enumerate}[(a)]
		\item[]
		
		\item Fix the parameters for the subordinator (inverse subordinator), under consideration. Choose  $\lambda>0$ and  order $k$ for the PPoK.
		\item  Fix the time $T$ for the time interval $[0,T]$ and choose $n+1$ uniformly spaced time points $0=t_0,t_1,\ldots,t_n=T$ with $h=t_2-t_1$.
		\item Simulate the values $W(t_i),1\leq i\leq n,$ of the subordinator (inverse subordinator) at $t_1,\ldots t_n,$ using the Algorithm 2--5 of \cite{TCFPP-pub} for respective subordinator (inverse subordinator).
		\item  Using the values $W(t_i),1\leq i\leq n,$ generated in Step (c), as time points, compute the number of events of the PPoK $\{N^{(k)}(W(t_i))\},1\leq i\leq n,$ using Algorithm \ref{simulation-fpp}.
	\end{enumerate}
\end{algorithm}
\noindent Let $\lambda=1.2$ and $T=10$ be fixed for the simulated sample paths presented in this section below. 
\begin{figure}[htb]
	\begin{subfigure}{0.5\textwidth}
		\caption{Parameters:  $k=3$}
		\includegraphics[width=0.8\linewidth]{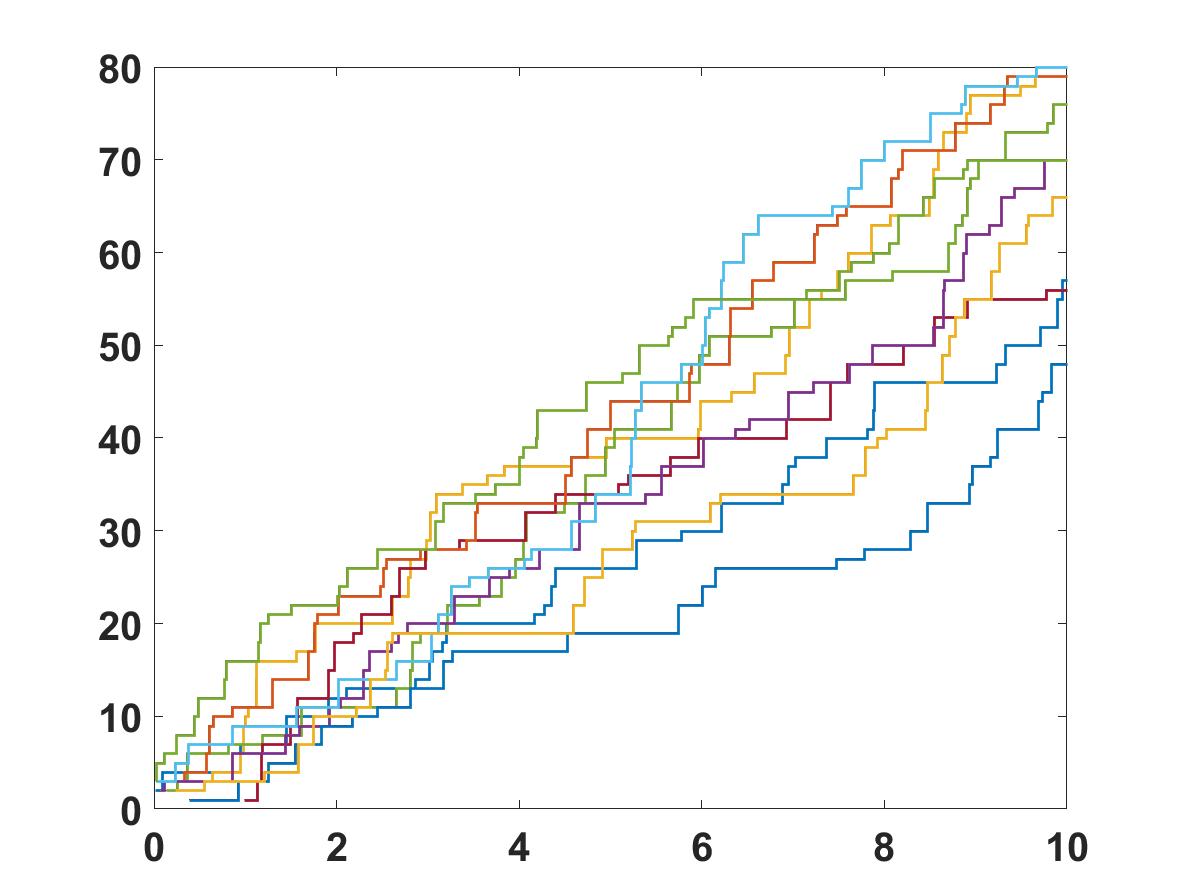}	\end{subfigure}~~
	\begin{subfigure}{0.5\textwidth}
		\caption{Parameters:  $k=5$}
		\includegraphics[width=0.8\linewidth]{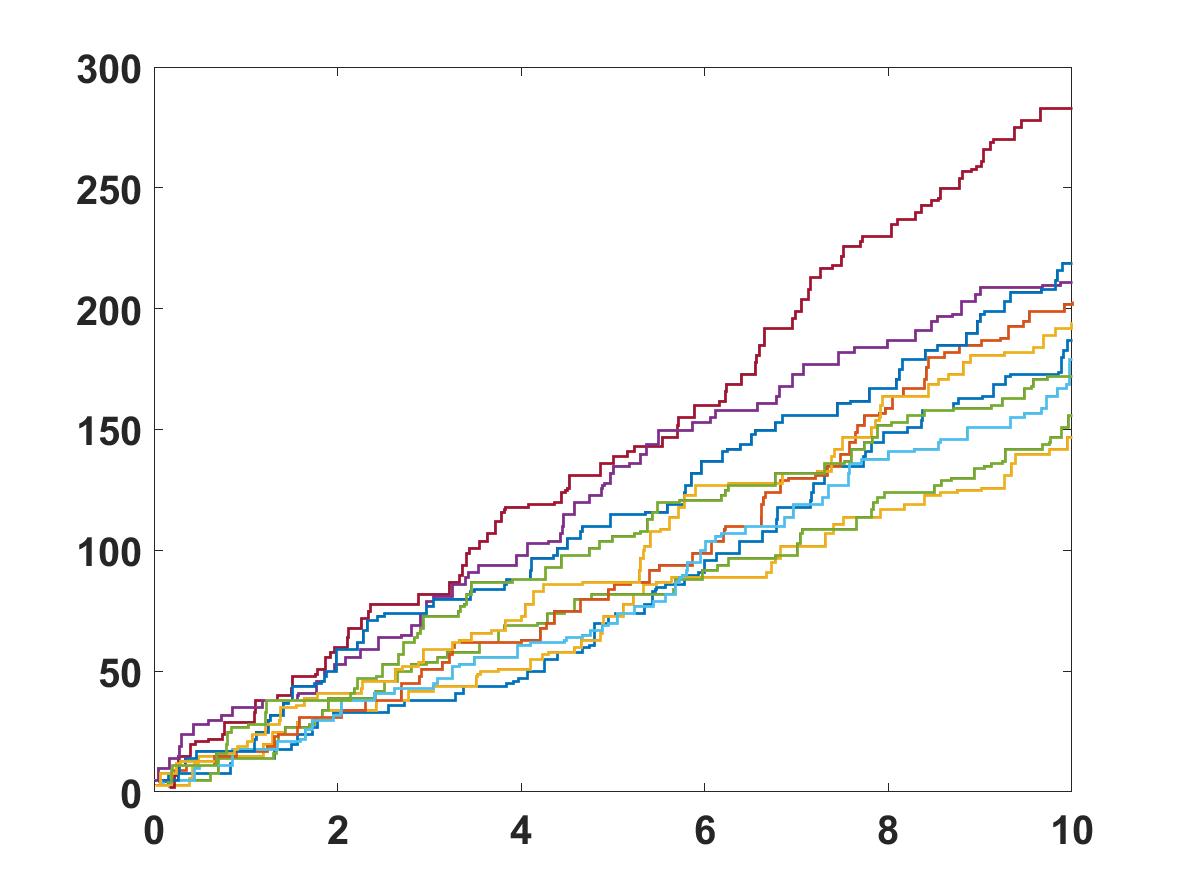}
	\end{subfigure}
	\caption{\label{fig:ppok} Ten simulated sample paths of the PPoK process for order {\tiny(A)} $k=3$, and {\tiny(B)} $k=5$}
\end{figure}
% The files are not sourced from the Master file but from the individual file
\begin{figure}[!htb]
	\begin{subfigure}{0.5\textwidth}

		\includegraphics[width=0.8\linewidth]{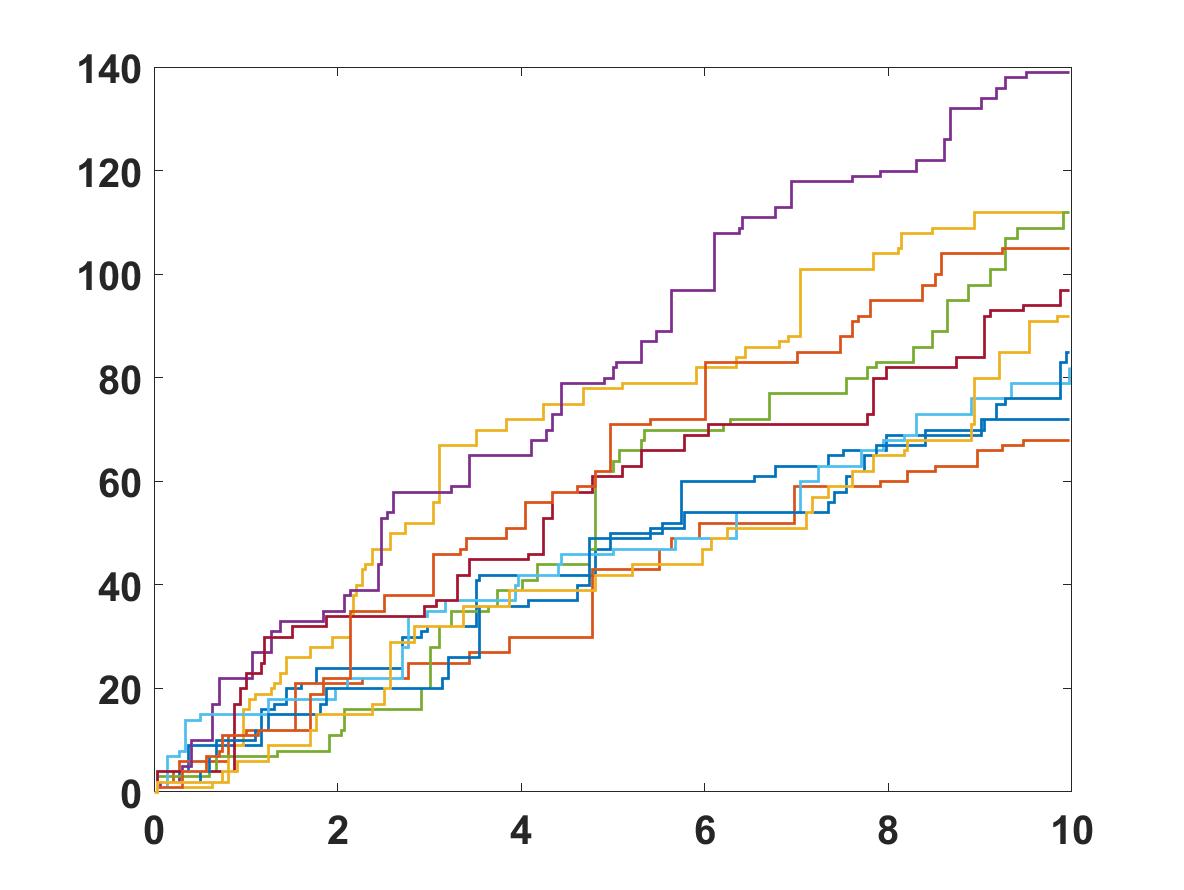}	
			\caption{Parameters: $\alpha=3.0,p=4.0,k=3$}	\end{subfigure}~~
	\begin{subfigure}{0.5\textwidth}

				\includegraphics[width=0.8\linewidth]{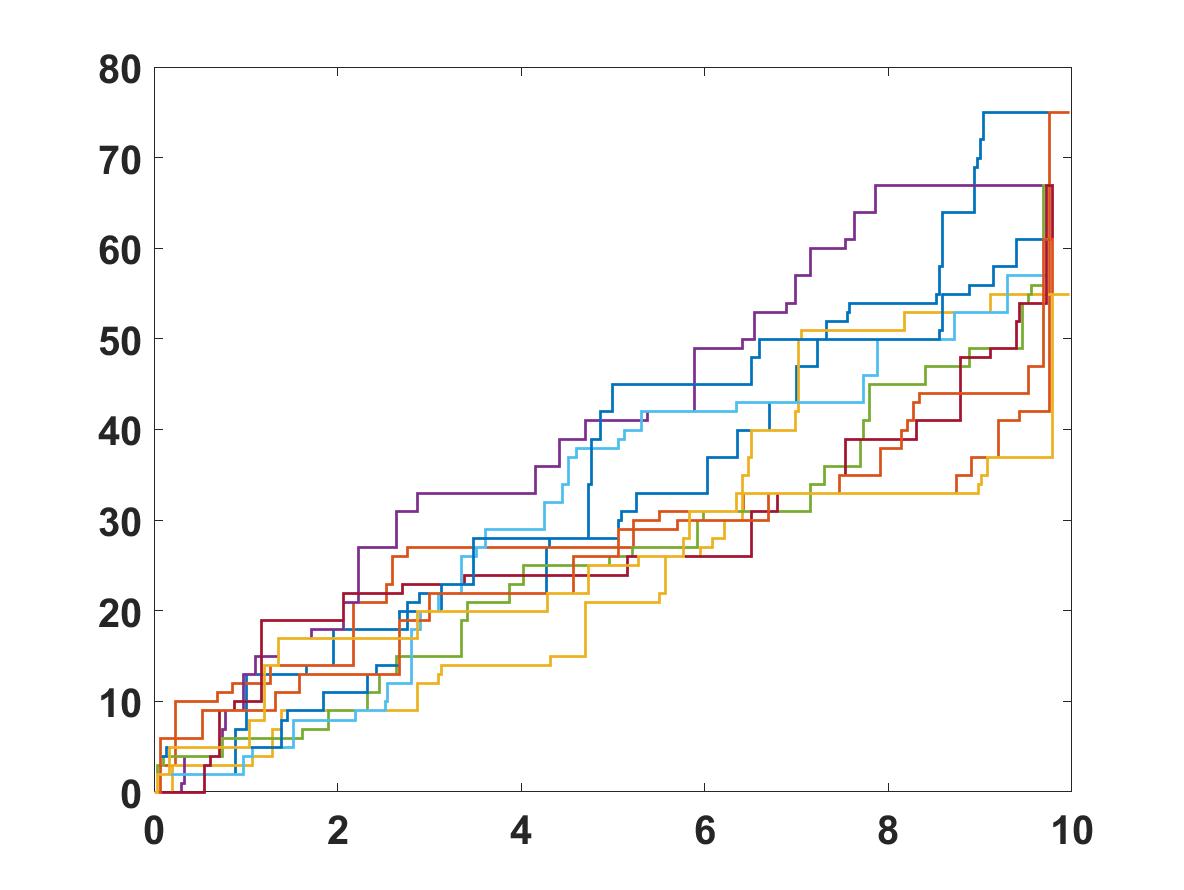}		\caption{Parameters: $\alpha=3.0,p=4.0,k=3$}
	\end{subfigure}
	\caption{\label{fig:ppokgamma}Ten simulated sample paths of time-changed PPoK with {\tiny(A)} gamma subordinator, and {\tiny(B)} inverse gamma subordinator.}
\end{figure}

\begin{figure}[!htb]
	\begin{subfigure}{0.5\textwidth}

		\includegraphics[width=0.8\linewidth]{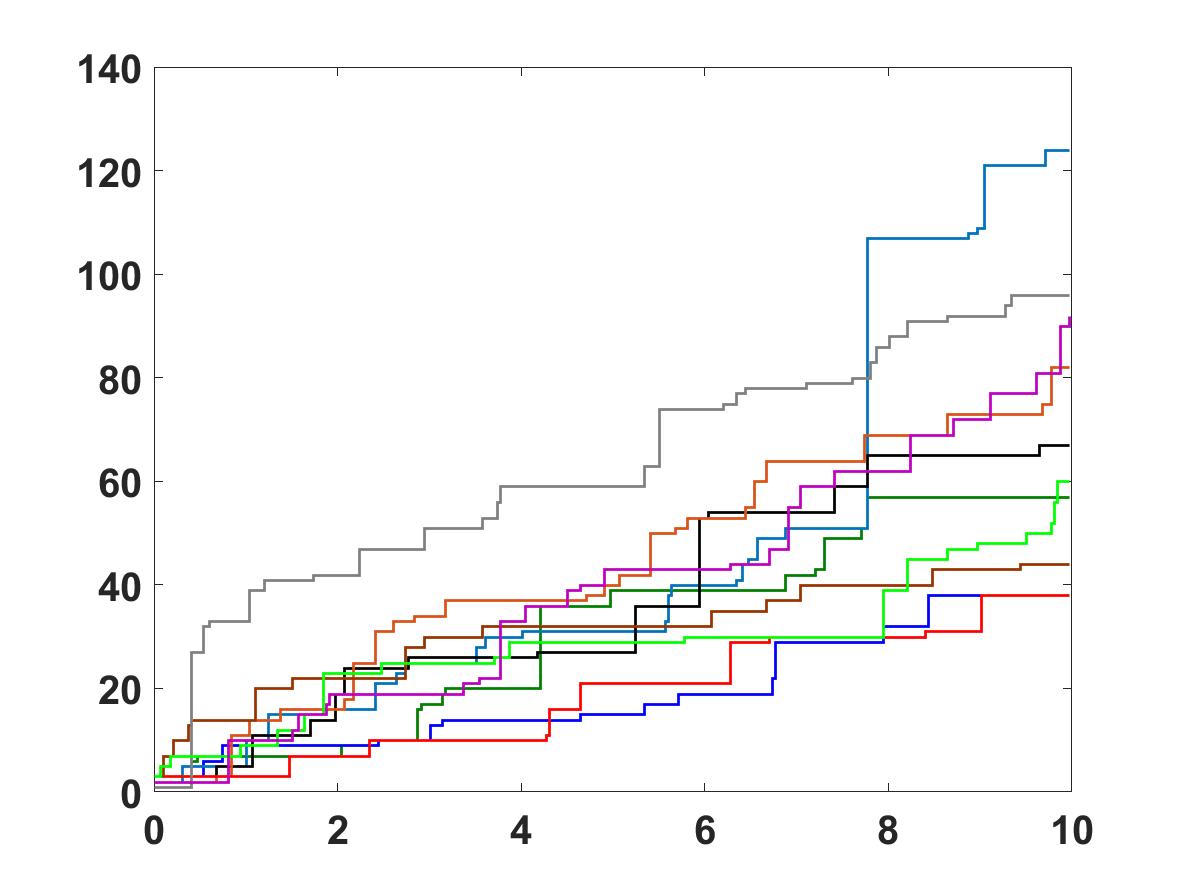}
	\caption{Parameters: $\gamma=1,\delta=1,k=3$}	\end{subfigure}~~
	\begin{subfigure}{0.5\textwidth}

		\includegraphics[width=0.8\linewidth]{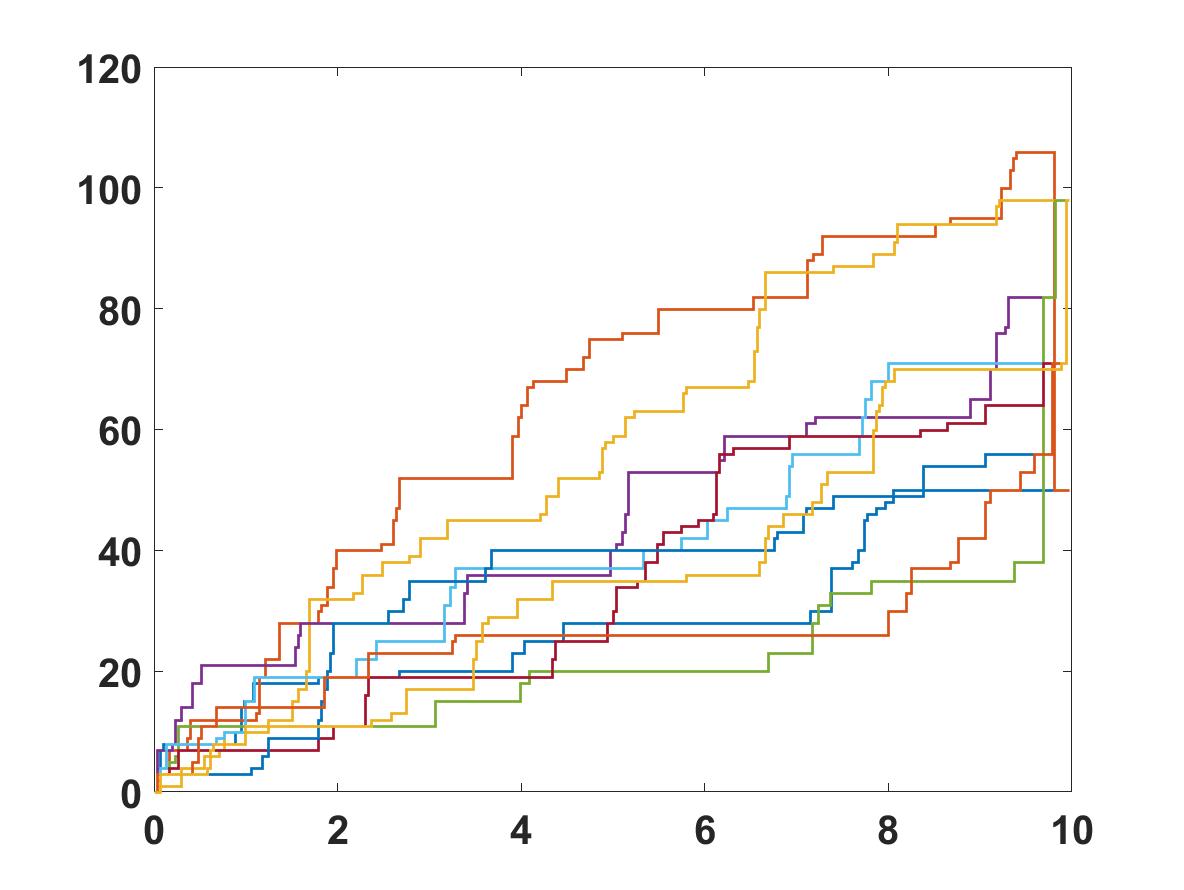}
		\caption{Parameters: $\gamma=1,\delta=1,k=3$}
	\end{subfigure}
	\caption{\label{fig:ppokIG}Ten simulated sample paths time-changed PPoK with {\tiny(A)} inverse Gaussian subordinator, and {\tiny(B)} inverse of inverse Gaussian subordinator. }
\end{figure}
\FloatBarrier
\subsection*{Interpretation of plots}
The PPoK is interpreted as arrival coming in packets of size $k$. As it is clear from Figure \ref{fig:ppok}, as the packet size $k$ is increased from 3 to 5, the number of arrivals increased. The effect of time-change by subordinator in PPoK is clearly visible in Figure {\small\ref{fig:ppokgamma}{\tiny (A)} and {\small \ref{fig:ppokIG}{\tiny (A)}} as the arrival rate of the packets increases compared with Figure  {\small\ref{fig:ppok}{\tiny (A)}. While if we observe the effect of inverse subordinator in  {\small\ref{fig:ppokgamma}{\tiny (B)} and {\small \ref{fig:ppokIG}{\tiny (B)}}, we find that the waiting time between events are increased predominantly compared to Figure  {\small\ref{fig:ppok}{\tiny (A)}.
\section*{Acknowledgement}
\noindent		
The authors are grateful to Prof. Enzo Orsingher for several helpful comments and suggestions which improved the quality of the article.
\bibliographystyle{abbrv}
\bibliography{researchbib}

\def\cprime{$'$}
\begin{thebibliography}{10}

\bibitem{Allouba02}
H.~Allouba.
\newblock Brownian-time processes: the {PDE} connection. {II}. {A}nd the
  corresponding {F}eynman-{K}ac formula.
\newblock {\em Trans. Amer. Math. Soc.}, 354(11):4627--4637 (electronic), 2002.

\bibitem{Allouba-Zheng01}
H.~Allouba and W.~Zheng.
\newblock Brownian-time processes: the {PDE} connection and the half-derivative
  generator.
\newblock {\em Ann. Probab.}, 29(4):1780--1795, 2001.

\bibitem{appm}
D.~Applebaum.
\newblock {\em L\'evy Processes and Stochastic Calculus}, volume 116 of {\em
  Cambridge Studies in Advanced Mathematics}.
\newblock Cambridge University Press, Cambridge, second edition, 2009.

\bibitem{Baeum-Meersch-Nane09}
B.~Baeumer, M.~M. Meerschaert, and E.~Nane.
\newblock Brownian subordinators and fractional {C}auchy problems.
\newblock {\em Trans. Amer. Math. Soc.}, 361(7):3915--3930, 2009.

\bibitem{Barn-Niel97}
O.~E. Barndorff-Nielsen.
\newblock Normal inverse {G}aussian distributions and stochastic volatility
  modelling.
\newblock {\em Scand. J. Statist.}, 24(1):1--13, 1997.

\bibitem{Barn-Niel98}
O.~E. Barndorff-Nielsen.
\newblock Processes of normal inverse {G}aussian type.
\newblock {\em Finance Stoch.}, 2(1):41--68, 1998.

\bibitem{beghinejp2009}
L.~Beghin and E.~Orsingher.
\newblock Fractional {P}oisson processes and related planar random motions.
\newblock {\em Electron. J. Probab.}, 14:no. 61, 1790--1827, 2009.

\bibitem{bertoin}
J.~Bertoin.
\newblock {\em L\'evy processes}, volume 121 of {\em Cambridge Tracts in
  Mathematics}.
\newblock Cambridge University Press, Cambridge, 1996.

\bibitem{Bochner}
S.~Bochner.
\newblock Diffusion equation and stochastic processes.
\newblock {\em Proceedings of the National Academy of Sciences of the United
  States of America.}, 35(7):368--370, 1949.

\bibitem{Cahoy2010}
D.~O. Cahoy, V.~V. Uchaikin, and W.~A. Woyczynski.
\newblock Parameter estimation for fractional {P}oisson processes.
\newblock {\em J. Statist. Plann. Inference}, 140(11):3106--3120, 2010.

\bibitem{subordinator:fin4}
L.~Calvet, B.~Mandelbrot, and A.~J. Fisher.
\newblock A multifractal model of asset returns.
\newblock Working papers, HAL, 2011.

\bibitem{subordinator:fin1}
P.~K. Clark.
\newblock A subordinated stochastic process model with finite variance for
  speculative prices.
\newblock {\em Econometrica}, 41(1):135--155, 1973.

\bibitem{ContTan2004}
R.~Cont and P.~Tankov.
\newblock {\em Financial modelling with jump processes}.
\newblock Chapman \& Hall/CRC Financial Mathematics Series. Chapman \&
  Hall/CRC, Boca Raton, FL, 2004.

\bibitem{subordinator:fin3}
M.~M. Dacorogna, U.~A. Müller, R.~J. Nagler, R.~B. Olsen, and O.~V. Pictet.
\newblock A geographical model for the daily and weekly seasonal volatility in
  the foreign exchange market.
\newblock {\em Journal of International Money and Finance}, 12(4):413 -- 438,
  1993.

\bibitem{sub:phy4}
B.~Dybiec and E.~Gudowska-Nowak.
\newblock Subordinated diffusion and continuous time random walk asymptotics.
\newblock {\em Chaos: An Interdisciplinary Journal of Nonlinear Science},
  20(4):043129, 2010.

\bibitem{sub:phy2}
R.~Failla, P.~Grigolini, M.~Ignaccolo, and A.~Schwettmann.
\newblock Random growth of interfaces as a subordinated process.
\newblock {\em Phys. Rev. E}, 70:010101, Jul 2004.

\bibitem{sub:bio}
I.~Golding and E.~C. Cox.
\newblock Physical nature of bacterial cytoplasm.
\newblock {\em Phys. Rev. Lett.}, 96:098102, Mar 2006.

\bibitem{Hahn-Kobaya-Umarov11}
M.~G. Hahn, K.~Kobayashi, and S.~Umarov.
\newblock Fokker-{P}lanck-{K}olmogorov equations associated with time-changed
  fractional {B}rownian motion.
\newblock {\em Proc. Amer. Math. Soc.}, 139(2):691--705, 2011.

\bibitem{Poiss-order-k}
K.~Y. Kostadinova and L.~D.Minkova.
\newblock On the {P}oisson process of order $k$.
\newblock {\em Pliska Stud. Math. Bulgar.}, 22, 2012.

\bibitem{Kumar-TCPP}
A.~Kumar, E.~Nane, and P.~Vellaisamy.
\newblock Time-changed poisson processes.
\newblock {\em Statistics \& Probability Letters}, 81(12):1899 -- 1910, 2011.

\bibitem{Kumar-ITSS}
A.~Kumar and P.~Vellaisamy.
\newblock Inverse tempered stable subordinators.
\newblock {\em Statistics \& Probability Letters}, 103:134 -- 141, 2015.

\bibitem{lask}
N.~Laskin.
\newblock Fractional {P}oisson process.
\newblock {\em Commun. Nonlinear Sci. Numer. Simul.}, 8(3-4):201--213, 2003.
\newblock Chaotic transport and complexity in classical and quantum dynamics.

\bibitem{lrd2016}
A.~Maheshwari and P.~Vellaisamy.
\newblock On the long-range dependence of fractional poisson and negative
  binomial processes.
\newblock {\em J. Appl. Probab.}, 53:989--1000, 2016.

\bibitem{TCFPP-pub}
A.~Maheshwari and P.~Vellaisamy.
\newblock Fractional poisson process time-changed by l{\'e}vy subordinator and
  its inverse.
\newblock {\em Journal of Theoretical Probability}, Dec 2017.

\bibitem{Mandelbrot01}
B.~B. Mandelbrot.
\newblock Scaling in financial prices: I. tails and dependence.
\newblock {\em Quantitative Finance}, 1(1):113--123, 2001.

\bibitem{subordinator:fin2}
C.~Marinelli, S.~Rachev, and R.~Roll.
\newblock Subordinated exchange rate models: evidence for heavy tailed
  distributions and long-range dependence.
\newblock {\em Mathematical and Computer Modelling}, 34(9):955 -- 1001, 2001.

\bibitem{Meersch-Koz-Molz-Lu04}
M.~M. Meerschaert, T.~J. Kozubowski, F.~J. Molz, and S.~Lu.
\newblock {Fractional Laplace model for hydraulic conductivity}.
\newblock {\em Geophys. Res. Lett.}, 31:L08501, 2004.

\bibitem{mnv}
M.~M. Meerschaert, E.~Nane, and P.~Vellaisamy.
\newblock The fractional {P}oisson process and the inverse stable subordinator.
\newblock {\em Electron. J. Probab.}, 16:no. 59, 1600--1620, 2011.

\bibitem{sub:phy1}
M.~G. Nezhadhaghighi, M.~A. Rajabpour, and S.~Rouhani.
\newblock First-passage-time processes and subordinated schramm-loewner
  evolution.
\newblock {\em Phys. Rev. E}, 84:011134, Jul 2011.

\bibitem{Orsingher2012}
E.~Orsingher and F.~Polito.
\newblock Compositions, random sums and continued random fractions of poisson
  and fractional poisson processes.
\newblock {\em Journal of Statistical Physics}, 148(2):233--249, Aug 2012.

\bibitem{sfpp}
E.~Orsingher and F.~Polito.
\newblock The space-fractional {P}oisson process.
\newblock {\em Statist. Probab. Lett.}, 82(4):852--858, 2012.

\bibitem{OrsToa-Berns}
E.~Orsingher and B.~Toaldo.
\newblock Counting processes with bernstein intertimes and random jumps.
\newblock {\em Journal of Applied Probability}, 52(4):1028--1044, 2015.

\bibitem{phil1984}
A.~N. Philippou.
\newblock Poisson and compound poisson distributions of order k and some of
  their properties.
\newblock {\em Journal of Soviet Mathematics}, 27(6):3294--3297, Dec. 1984.

\bibitem{phili83-geo}
A.~N. Philippou, C.~Georghiou, and G.~N. Philippou.
\newblock A generalized geometric distribution and some of its properties.
\newblock {\em Statistics \& Probability Letters}, 1(4):171 -- 175, 1983.

\bibitem{sato}
K.~Sato.
\newblock {\em L\'evy {P}rocesses and {I}nfinitely {D}ivisible
  {D}istributions}, volume~68 of {\em Cambridge Studies in Advanced
  Mathematics}.
\newblock Cambridge University Press, Cambridge, 1999.
\newblock Translated from the 1990 Japanese original, Revised by the author.

\bibitem{Sato2001}
K.-i. Sato.
\newblock Subordination and self-decomposability.
\newblock {\em Statist. Probab. Lett.}, 54(3):317--324, 2001.

\bibitem{sub:eco}
H.~Scher, G.~Margolin, R.~Metzler, J.~Klafter, and B.~Berkowitz.
\newblock The dynamical foundation of fractal stream chemistry: The origin of
  extremely long retention times.
\newblock {\em Geophysical Research Letters}, 29(5):5--1--5--4.

\bibitem{Bernstein-book}
R.~L. Schilling, R.~Song, and Z.~Vondra{\v{c}}ek.
\newblock {\em Bernstein functions}, volume~37 of {\em de Gruyter Studies in
  Mathematics}.
\newblock Walter de Gruyter \& Co., Berlin, second edition, 2012.
\newblock Theory and applications.

\bibitem{sub:phy3}
A.~Stanislavsky and K.~Weron.
\newblock Two-time scale subordination in physical processes with long-term
  memory.
\newblock {\em Annals of Physics}, 323(3):643 -- 653, 2008.

\bibitem{Taqqu2010}
M.~Veillette and M.~S. Taqqu.
\newblock Numerical computation of first passage times of increasing {L}\'evy
  processes.
\newblock {\em Methodol. Comput. Appl. Probab.}, 12(4):695--729, 2010.

\bibitem{Kumar-Hitting}
P.~Vellaisamy and A.~Kumar.
\newblock First-exit times of an inverse gaussian process.
\newblock {\em Stochastics}, 90(1):29--48, 2018.

\bibitem{fnbpfp}
P.~Vellaisamy and A.~Maheshwari.
\newblock Fractional negative binomial and {P}olya processes.
\newblock {\em Probab. Math. Statist.}, 38(1):77--101, 2018.

\end{thebibliography}
\end{document}